\documentclass[11pt,twoside]{amsart}
\usepackage[all]{xy}
\usepackage {amssymb,latexsym,amsthm,amsmath,mathtools, multirow, amsfonts, longtable, eqnarray, hyperref,comment,amscd}
\usepackage{enumitem,color}

\hypersetup{
	colorlinks=true,
	linkcolor=blue,
	filecolor=magenta,      
	urlcolor=cyan,
	citecolor=red,
}

\long\def\symbolfootnote[#1]#2{\begingroup%
	\def\thefootnote{\fnsymbol{footnote}}\footnote[#1]{#2}\endgroup}
\newcommand{\C}{\ensuremath{\mathfrak{C}}}
\newcommand{\Z}{\ensuremath{\mathcal{Z}}}

\newcommand{\F}{\mathbb F}

\newcommand{\0}{\textbf{0}}

\newcommand{\gen}[1]{\left < #1 \right >}

\makeatletter
\def\imod#1{\allowbreak\mkern10mu({\operator@font mod}\,\,#1)}
\makeatother

\newtheorem{theorem}{Theorem}[section]
\newtheorem{lemma}[theorem]{Lemma}

\newtheorem{corollary}[theorem]{Corollary}
\newtheorem{proposition}[theorem]{Proposition}
\newtheorem*{theorem*}{Theorem}
\theoremstyle{definition}
\newtheorem{definition}[theorem]{Definition}
\newtheorem{remark}[theorem]{Remark}
\newtheorem{question}[theorem]{Question}
\newtheorem{example}[theorem]{Example}

\numberwithin{equation}{section}
\newcommand{\ignore}[1]{}

\newcommand{\mynote}[1]{}
\begin{document}

\title[Nilpotent Lie algebras with two centralizer dimensions]{Nilpotent Lie algebras with two 
centralizer dimensions over a finite field}
\author{Rijubrata Kundu}
\address{Indian Institute of Science Education and Research Mohali, Knowledge City, Sector 81, 
Mohali 140 306, India}
\email{rijubrata8@gmail.com}
\author{Tushar Kanta Naik}
\address{School of Mathematical Sciences, National Institute of Science Education and Research, Bhubaneswar, An OCC of Homi Bhabha National Institute, P. O. Jatni, Khurda 752050, Odisha, India.}
\email{mathematics67@gmail.com/tushar@niser.ac.in}
\author{Anupam Singh}
\address{Indian Institute of Science Education and Research Pune, Dr. Homi Bhabha Road, Pashan, Pune 411 008, India}
\email{anupamk18@gmail.com}
\thanks{Rijubrata Kundu thanks IISER Mohali for the Institute Post-Doctoral Fellowship. Tushar Kanta Naik would like to acknowledge partial support from NBHM via grant 0204/3/2020/R\&D-II/2475. Anupam Singh is funded by SERB through CRG/2019/000271 for this research.}
\today

\begin{abstract}
A result of Barnea and Isaacs states that if $L$ is a finite dimensional nilpotent Lie algebra 
with exactly two distinct centralizer dimensions, then nilpotency class of $L$ is either $2$ or $3$. 
In this article, we classify all such finite dimensional $3$-step nilpotent Lie algebras over a
finite field.
\end{abstract}

\maketitle

%%%%%%%%%%%%%%%%%%%%%%
\section{Introduction}\label{introduction}

Let $L$ be a finite dimensional nilpotent Lie algebra. A Lie algebra $L$ is said to be of {\bf 
breadth type} $(0=n_0, n_1, n_2,\cdots, n_r)$ if $n_i$ are the distinct breadths (codimension of 
centralizers) of elements of $L$ written in increasing order (see~\cite{kns}). This article deals 
with the classification of such algebras of breadth type $(0, n)$ (that is, two centralizer 
dimensions) over a finite field where $n\geq 1$ is an integer. Let $L$ be a non-abelian such algebra 
over $\mathbb F_q$ of breadth type $(0, n)$. Barnea and Isaacs (see Theorem B,~\cite{bi}) proved 
that such Lie algebras have nilpotency class at most $3$. The Lie algebras $L$ of breadth type 
$(0,1), (0,2)$ and $(0,3)$ (in odd characteristics) over finite fields are classified in~\cite{kns} 
following the work in~\cite{kms,swk,re}. It turns out that the Lie algebras of type $(0,1)$ and 
$(0,3)$ are of class $2$. Whereas the type $(0,2)$ could be of class $2$ and $3$ both (see Theorem 
1.1, \cite{kns}). Furthermore, there is a unique type $(0,2)$ stem Lie algebra of class $3$ (in odd 
characteristics). Thus, it raises the following questions: 
\begin{question}
\begin{enumerate}
\item[(a)]  When $n$ is odd, is it the case that any such Lie algebra of type $(0, n)$ of class $2$? 
\item[(b)] Can we give a complete classification of such algebras of class $3$ (which would exist when $n$ is even)? 
\item[(c)] Can we give the classification of class $2$ such algebras? 
\end{enumerate}
\end{question}
\noindent In this article we answer the first two questions completely, and the third part remains 
to be solved.

The following two theorems are the main results of this paper.
\begin{theorem}\label{odd_dimension_three_step_nilpotent}
Let $L$ be a finite dimensional $3$-step nilpotent Lie algebra of breadth type $(0, n)$ over 
$\F_{q}$, and $q\geq 3$. Then $n$ is necessarily an even integer.
\end{theorem}
\noindent Thus, when $n$ is odd the nilpotency class of such Lie algebras is $2$. Furthermore, when 
$n$ is even, we have the following:

%	\begin{theorem}\label{even_dimension_three_step_nilpotent}
%		Suppose that $n \geq 2$ is an even integer. Let $L$ be a finite dimensional 3-step nilpotent Lie algebra of breadth type $(0,n)$ over $\F_q$ where $q\geq 3$. Then there exists a unique (up to isomorphism) Lie algebra $\mathcal{V}$ such that $Z(\mathcal{V})\leq \mathcal{V}'$, $\mathcal{V}$ has breadth type $(0,n)$ and $L\cong \mathcal{V}\oplus A$ where $A$ is an abelian Lie algebra over $\F_q$.
%	\end{theorem}
	
\begin{theorem}\label{even_dimension_three_step_nilpotent}
Suppose that $n \geq 2$ is an even integer, and $q\geq 3$. Then there exists a unique  finite 
dimensional $3$-step nilpotent stem Lie algebra of breadth type $(0, n)$ over $\F_q$.
\end{theorem}
\noindent This completes the classification of all 3-step nilpotent Lie algebras of breadth type $(0,n)$ (see Section 2.1, \cite{kns}). 

As a 
consequence of the proof of the Theorem~\ref{odd_dimension_three_step_nilpotent}, we obtain that if $L$ is a 3-step nilpotent stem 
Lie algebra of breadth type $(0,2m)$ then $\dim(L/L')=2m$ and $\dim(L'/Z(L))=m$ (see 
Proposition~\ref{dimension_L_over_L'} and Corollary~\ref{dimension_derived_over_center}). 
Further, we introduce the notion of semifield Lie algebras over finite fields (which are certain 
$2$-step Camina Lie algebras). We connect these Lie algebras to our problem and get further 
important structural 
results amongst which we prove that $L/Z(L)\cong \mathrm{U}_3(q^m)$ as $\F_q$-Lie algebras (see 
Theorem~\ref{central_quotient_recognition}) and $Z(L)=\gamma_3(L)$ (see 
Proposition~\ref{dimension_center}). These structural properties lead us to a complete 
classification of such Lie algebras given in Theorem~\ref{even_dimension_three_step_nilpotent}.

Suppose $m\geq 1$. Let $\mathfrak{g}_m$ be the Lie algebra whose underlying vector space is 
$\F_{q^m}^5$ over $\F_q$ (of $5m$ dimension), and the Lie bracket $[,]$ defined by
$$[(a,b,c,d,e), (x,y,z,u,v)] = (0, 0, ay-bx, ay-bx+2cx-2az, ay-bx+2bz-2cy).$$
\noindent Combining with the Theorem~\ref{even_dimension_three_step_nilpotent}, we obtain the 
following:
\begin{corollary}\label{Corr}
Let $L$ be a finite dimensional $3$-step nilpotent Lie algebra over a finite field $\F_q$ of 
characteristic not equal to $2$. For $n=2m\geq 2$ even, if $L$ has breadth type $(0,2m)$ then $L\cong 
\mathfrak{g}_m$.
\end{corollary}
\noindent We prove Theorem~\ref{odd_dimension_three_step_nilpotent} in Section~\ref{section 2} and 
Theorem~\ref{even_dimension_three_step_nilpotent} in Section~\ref{section 5}. We mention that these 
results are a Lie-algebra analog of some similar results in $p$-groups (see~\cite{nky}). 

Our method depends on certain results which involve finite fields. This leaves us to wonder if 
Theorem~\ref{odd_dimension_three_step_nilpotent} is true over arbitrary field. In general, can we 
classify finite dimensional 3-step nilpotent Lie algebras of breadth type $(0,n)$ over an arbitrary 
field? Nilpotent Lie algebras of small dimensions are classified (see for example~\cite{dg, jnp, 
se, ro, sc,sgs}). In higher dimensions the classification problem is difficult. We hope our results would pave a 
way to look at this more general problem.

\subsection{Notations}
We set some notations that will be used throughout the paper. A Lie algebra is assumed to be finite 
dimensional over a field $k$. Let $\F_q$ denote the finite field with $q$ elements. For a Lie 
algebra $L$, the center of $L$ is denoted by $Z(L)$ and the derived subalgebra $[L, L]$ by $L'$. A 
$c$-step nilpotent Lie algebra is a nilpotent Lie algebra of nilpotency class $c$. For a lie algebra 
$L$, denote $L^{0}=L, L^1=[L,L]=L'$ and $L^i=[L,L^{i-1}]$ for all $i\geq 1$. Further we will denote 
$L^i$ by $\gamma_{i+1}(L)$ for every $i\geq 0$. Finally, $\mathrm{U}_n(q^m)$ denotes the Lie algebra 
(over $\F_q$) of $n\times n$ strictly upper triangular matrices with entries in $\F_q$.

\iffalse
The paper is organized as follows: In Section~\ref{section 2} we prove 
Theorem~\ref{odd_dimension_three_step_nilpotent}. In Section~\ref{section 3}, we show that there 
exists a 3-step nilpotent Lie algebra over $\F_q$ ($q$ odd) of breadth type $(0,2m)$ where $m\geq 
1$. This is precisely the Lie algebra $\mathfrak{g}_m$ which has been defined above. In 
Section~\ref{section 4} we prove an important recognition theorem (see 
Theorem~\ref{Characterization_Camina_Lie_algebra}) involving 2-step nilpotent Camina Lie algebras 
over finite fields. This result will play a crucial role in the proof of 
Theorem~\ref{even_dimension_three_step_nilpotent}. We prove this theorem by introducing the 
important notion of semifield Lie algebras  (analogous to semifield groups). Finally, in 
Section~\ref{section 5}, we prove Theorem~\ref{even_dimension_three_step_nilpotent}.
\fi 
%%%%%%%%%%%%%%%%%%%%%%
%\section{Proof of Theorem~\ref{odd_dimension_three_step_nilpotent}}\label{section 2}

\section{Structure of $3$-step nilpotent Lie algebra of type $(0,n)$}\label{section 2}

Throughout this section, let $L$ be a finite dimensional nilpotent Lie algebra over a finite field 
$\F_q$ with $q$ elements, where $q\geq 3$. Before moving any further, we mention that as a 
consequence of isoclinism of Lie algebras (see Section 2, \cite{kns} for details), it is enough to 
work with a stem (or, pure) Lie algebra. A Lie algebra $L$ over a field $k$ is said to be stem (or, 
pure) if $Z(L)\leq L'$. This key reduction is due to the fact that isoclinic Lie algebras have the 
same breadth type. Further, any isoclinism class contains a stem Lie algebra.
	
We start with a lemma.
\begin{lemma}\label{lem2}
Let $L$ be a 3-step nilpotent stem Lie algebra of breadth type $(0,n)$ for some $n\geq 1$. Then 
$\dim\;(L'/Z(L))<n$.
\end{lemma}

\begin{proof}
Let $\dim\;(L/L')=m$. Since the nilpotency class of $L$ is $3$, $L'$ is abelian. Thus for any $y 
\in L' \setminus Z(L)$, we have $L'\leq C_L(y)$. Since $\dim\;(L/C_L(y))=n$, we conclude that $n\leq 
m$. Suppose that $\dim\;(L'/Z(L))=k$ with $k \geq n$. Consider the following set:
\[ X=\{(\langle x+L'\rangle, \langle y+Z(L)\rangle ) \mid x \in L \setminus L', \;y \in 
L'\setminus Z(L), \;[x,y]=0\}. \]

\noindent Let $x+L'=x_1+L'$ and $y+Z(L)=y_1+Z(L)$. Then there exists $x_2\in L'$ and $y_2\in 
Z(L)$ such that $x=x_1+x_2$ and $y=y_1+y_2$. Therefore,
$$[x,y]= [x_1+x_2,y_1+y_2] = [x_1,y_1]+[x_2,y_1]+[x_1+x_2,y_2]=[x_1,y_1].$$ 
Hence $X$ is well defined. 

Since $\dim\;(L'/Z(L))=k$, there are $\frac{q^k-1}{q-1}$ subalgebras of dimension $1$ in 
$L'/Z(L)$. Since $L'$ is abelian, for each $y \in L' \setminus Z(L)$, $L' \subseteq C_L(y)$ and 
$\dim\;(C_L(y)/L')=m-n$. As there are $\frac{q^{m-n}-1}{q-1}$ subalgebras of dimension $1$ in  
$C_L(y)/L'$, we get 
$$|X|=\frac{(q^k-1)(q^{m-n}-1)}{(q-1)^2}.$$

On the other hand, fix $x \in L \setminus L'$. Set $C_{L'}(x):=C_L(x)\cap L' $. We know that 
$\dim\;([x, L']) \leq \dim\;([x,L]) = n$. Suppose $\dim\;([x,L']) = \dim\;([x,L])=n$, that is, 
$[x,L'] = [x,L]$. Let $\{x_1,\ldots, x_m\}$ be a minimal generating set of $L$ over $L'$. Then 
$[x,x_i] = [x,h_i]$ for some $h_i\in L'$ and  hence $[x,x_i-h_i]=0$ for $1 \le i \le m$. Hence 
$\{x_1-h_1,\ldots, x_m-h_m\}$ is also a generating set for $L$ over $L'$, which centralizes $x$. 
This implies that $x \in Z(L)$, a contradiction. Therefore, $\dim\;([x,L'])< \dim\;([x,L]) = n$ and  
$\dim\;(C_{L'}(x)/Z(L)) \geq k+1-n \geq 1$. Thus, there are at least $(q^{k+1-n}- 1)/(q-1)$ 
subspaces of dimension $1$ in $L'/Z(L)$ with $[x,y]=0$. Hence, we  get
		$$|X| \geq \frac{(q^m-1)(q^{k+1-n}-1)}{(q-1)^2}.$$
		Comparing the size of $X$, we get
		$$(q^m - 1)(q^{k+1-n}- 1) \leq (q^{m-n} - 1)(q^k- 1),$$
		which on simplification gives
		$$q+q^{-k}+q^{n-m} \leq 1+q^{1-m}+q^{n-k}<3,$$
		a contradiction to our choice of $q$. Hence  dim$(L'/Z(L))<n$.
	\end{proof}
	
	\begin{definition}[\textbf{Relative breadth}]
		Let $L$ be a finite dimensional Lie algebra, $I$ be a non-trivial ideal of $L$ and $x \in L$. Then the {\it relative breadth} (with respect to $I$), $b_I(x)$ of $x$  is defined as rank of the adjoint operator 
		$$ad_{x}\mid _I: I\rightarrow I\;\; \text{given by}\;\; y\mapsto [x, y].$$
		
		The relative breadth of $I$ in $L$ is defined as $$b_I(L) = \text{max}\{b_I(x) \mid x \in L\}.$$
	\end{definition}
	
	\begin{lemma}\label{lem3}
		Let $L$ be a 3-step nilpotent stem Lie algebra of breadth type $(0,n)$ for some $n\geq 1$. Suppose  $B_{L'}=\{x \in L \mid C_{L'}(x)=Z(L)\}$. Then $\langle B_{L'}\rangle = L.$
	\end{lemma}
	
	\begin{proof}
		Let dim$(L/L')=m$ for some natural number $m$. As in the proof of Lemma~\ref{lem2}, we have  $n \le m$. Let dim$(L'/Z(L))=k$. Then, by  Lemma~\ref{lem2} we have  $k \leq n-1$. Define
		$$S:=\{x \in L \mid [x,y]=0\; \text{ for some element } y \in L'\setminus Z(L)\}.$$
		Note that $\displaystyle S = \bigcup C_L(h)$, where union is taken over all one dimensional subspaces $\langle h+L'\rangle$ in $L'/Z(L)$. Since $L'$ is abelian, for each $a\in L'\setminus Z(L)$, $L' \leq C_L(a)$ and $\dim\;(C_L(a))= \dim L'+(m-n)$. Thus
		\begin{equation*}
			|S| \leq |L'|q^{m-n}\frac{(q^k-1)}{q-1} < |L'|q^{m-n+k} \leq
			|L'|q^{m-1}<|L|.
		\end{equation*}
		Consequently, $S$ is a proper subspace of $L$, and therefore there exists an element $x \in L \setminus L'$ such that $C_{L'}(x)=Z(L)$. We have
		$$B_{L'}=\{x\in L \mid C_{L'}(x)=Z(L)\} = L \setminus S.$$
		If $\langle B_{L'}\rangle= W<L$, then $|W| \leq |L'|q^{m-1}$. Thus we have
		\begin{align*}
			|L| = |S \cup B_{L'}|\leq |L'|q^{m-1} + |L'|q^{m-1} < |L'|q^m=|L|,
		\end{align*}
		which is a contradiction. Hence  $\langle B_{L'} \rangle = L$, which completes the proof.
	\end{proof}
	
	The following determines the dimension of $L/L'$.
	
	\begin{proposition}\label{dimension_L_over_L'}
		Let $L$ be a 3-step nilpotent stem Lie algebra of breadth type $(0,n)$ for some $n\geq 1$. Then $\dim\;(L/L')=n$.
	\end{proposition}
	
	\begin{proof}
		When $n=1$, there is nothing to prove. When $n=2$ the result follows easily from Theorem 5.1, \cite{kms}. So we assume that $n>2$. Since $L'$ is abelian $\dim\;(L/L')\geq n$. If  dim$(L/L')= n$ we are done. Otherwise, let $\dim\;(L/L')\geq n+m$ with $m>0$.  We now proceed to get a contradiction. Let dim$(L'/Z(L)) =k$, where $k \leq n-1$ (by Lemma \ref{lem2}). We complete the proof in several steps.
		
		\medskip
		
		\noindent \textbf{\underline{Claim $1$}:} {\bf For all $h \in L' \setminus Z(L)$, $C_L(h)=C_L(L')$.}

		\noindent It follows from Lemma \ref{lem3} that we can choose an element $x_1 \in L$ such that
		$$C_{L'}(x_1)=C_{L}(x_1) \cap L' =Z(L).$$
		Let $\{x_1, x_2, \ldots,x_{m+k},\ldots, x_{m+n}\}$ be a minimal generating set for $L$ such that $C_L(x_1)=\langle x_1, x_2, \dots, x_{m+k},\; Z(L)\rangle$.
		
		If $\{[x_1,x_i] \mid m+k< i\leq m+n \}\subseteq Z(L)$, then $[x_1,x_i] \in Z(L)$ for all $i$ with $1 \le i \le m+n$. By an application of Jacobi identity, we get $[x_r,x_s] \in C_L(x_1) \cap L'=Z(L)$ for all $r,s$ with $1\le r,s\le m+n$, that is, $L' \subseteq Z(L)$, a contradiction. Thus, without loss of generality, we can assume that $[x_1,x_{m+n}] \notin Z(L)$. Consider the subspace
		$$U=\langle [x_1,x_{m+n}], [x_2,x_{m+n}],\ldots, [x_{m+k}, x_{m+n}],Z(L)\rangle.$$
		Since $\dim\;(L'/Z(L) = k$, among the Lie brackets $[x_i,x_{m+n}]$ ( $1\leq i\leq m+k$), at most $k$ are independent  modulo $Z(L)$. We can assume without loss of generality that for some integer $l$ with $1 \le l \le k$, $[x_1,x_{m+n}], [x_2,x_{m+n}],\ldots, [x_{l}, x_{m+n}]$ are independent modulo $Z(L)$, and generate $U$ along  with $Z(L)$.
		
		\medskip
		
		We now show that $l=k$. For any $t$ with $l<t\leq m+k$, we have 
		$$[x_t,x_{m+n}]=i_1[x_1,x_{m+n}] + i_2[x_2,x_{m+n}] + \cdots + i_l[x_l,x_{m+n}]+Z(L).$$
		Let $x_t'=x_t+ (-i_1)x_1 + \cdots + (-i_l)x_l$. Then $[x_t',x_{m+n}] \in Z(L)$. Thus with $C_L(x_1)=\langle x_1$, $x_2$, $\ldots$, $x_{m+k}$, $Z(L)\rangle$, we can assume (modifying $x_t$ by $x_t'$ if necessary) that,
		\begin{equation}\label{eqse3-a}
			[x_t,x_{m+n}] \in Z(L)\;\; \text{for}\;\; l<t\leq m+k.
		\end{equation}
		Since $[x_1,x_i]=[x_1,x_j]=0$ for any $i,j\leq m+k$, using Jacobi identity we get
		$$[x_j,x_i]\in C_L(x_1)\cap L'=Z(L).$$ 
		In particular, for $l<t\leq m+k$, $[x_t,x_i]\in Z(L)$. Consequently by \eqref{eqse3-a} we get
		$$[x_t,[x_i,x_{m+n}]]=0  \hskip5mm (l<t\leq m+k, \,\,1 \le i\leq l).$$
		Thus for any $i$ with $1 \le i \leq l$, $\langle x_{l+1},\ldots, x_{m+k}, L'\rangle \le C_L([x_i,x_{m+n}])$. So
		$$n=\dim\;(L/C_L([x_i,x_{m+n}])) \leq m+n-(m+k-l) =n-k+l\leq n.$$
		Consequently, we get
		\begin{enumerate}
			\item $k=l$,
			\item $U=\langle [x_1,x_{m+n}], [x_2,x_{m+n}],\ldots, [x_k,x_{m+n}], Z(L)\rangle=L'$, and 
			\item $C_L([x_i,x_{m+n}]) =\langle  x_{k+1},x_{k+2}\ldots,x_{k+m}, L'\rangle:= W$ (say), for all $1 \le i\leq k$.
		\end{enumerate}
		\vspace*{0.1cm}
		By (2) and (3), we get  $C_L(h)=W = C_L(L')$ for all $h\in L'\setminus Z(L)$. \textbf{This proves  Claim 1}.
		
		\medskip
		
		Observe that by Lemma \ref{lem3},
		$$L \setminus W=\{x \in L \mid C_{L'}(x)=Z(L)\}=B_{L'},$$ which implies  that $C_W(x) \cap L' = Z(L)$ for all $x \in L \setminus W$.
		\vskip3mm
		
		\noindent{\bf \underline{Claim $2$:}} {\bf For any $y \in L \setminus W$, there exist elements $h_1,h_2,\ldots,h_m \in L'$ such that
			$$\frac{\langle x_{k+i}+h_i: 1\leq i \leq m \rangle + Z(L)}{Z(L)} = \frac{C_L(y)\cap W}{Z(L)}$$ and is of dimension $m$.}
		
		\vskip2mm
		Let $y \in L \setminus W$ be any arbitrary element. Then $C_{L'}(y)=Z(L)$. Consider a minimal generating set $Y:= \{y=y_1, y_2, \ldots, y_{m+k}, \ldots, y_{m+n}\}$ of $L$ such that $C_L(y_1)=\langle y_1,y_2,\dots,y_{m+k}, Z(L)\rangle.$ As in Claim 1, we can modify (if necessary) the elements $y_{k+1}, \dots,$ $y_{m+k}$ so that 
		$$\langle y_{k+1},y_{k+2}, \dots, y_{m+k}\rangle +L'= W =\langle x_{k+1},x_{k+2},\dots,x_{m+k}\rangle+ L'.$$
		
		Consequently, for $ 1 \le i \le m$, it follows that $x_{k+i} = (\sum_j a_{ij}y_{k+j}) - h_i$ for some set of scalars $a_{ij}\in \F_q$ and some element $h_i \in L'$. Hence $x_{k+i} + h_i \in C_L(y)$, which shows that
		\begin{equation}\label{eqse3-0a}
			\frac{\langle x_{k+i}+h_i: 1\leq i \leq m \rangle + Z(L)}{Z(L)} \leq  \frac{C_L(y)\cap W}{Z(L)}.
		\end{equation}
		Note that for $1 \le i, j \le m$, $[x_{k+i} + h_i,x_{k+j}+h_j] \in Z(L)$. Thus
		\begin{equation}\label{eqse3-0b}
			\dim\;\left[\frac{\langle x_{k+i}+h_i: 1\leq i \leq m \rangle + Z(L)}{Z(L)} \right]= m.
		\end{equation}
		By the choice of $y_{k+1}, \dots,y_{m+k}$, it follows that
		$$\langle y_{k+1}, \dots, y_{k+m}\rangle + Z(L)=  C_L(y)\cap W,$$  and hence  $\dim\;\left[\frac{C_L(y)\cap W}{Z(L)}\right]= m$. \textbf{This along with  \eqref{eqse3-0a} and \eqref{eqse3-0b} proves Claim 2}.
		
		\medskip
		
		\noindent\textbf{\underline{Claim 3}.} \textbf{The cardinality of $$S:= \left\{(\langle x+Z(L)\rangle, \langle y+Z(L) \rangle) \mid x\in L\setminus W,\,\, y\in W\setminus L',[x,y]= 0\right\}$$ is $\frac{q^{m+k}(q^n-1)(q^m-1)}{(q-1)^2}$}.
		
		\medskip
		
		It is easy to see that $L/Z(L)$, $W/Z(L)$ and $L'/Z(L)$ have $\frac{q^{n+m+k}-1}{q-1}$, $\frac{q^{m+k}-1}{q-1}$, and $\frac{q^k-1}{q-1}$ number of subspaces of dimension 1 respectively. Consequently
		\begin{equation}\label{eqse3-1}
			\mid \{ \langle x+Z(L)\rangle \mid x \in L\setminus W \} \mid = \frac{q^{m+k}(q^n-1)}{q-1}
		\end{equation}
		and
		\begin{equation}\label{eqse3-2}
			\mid \{ \langle y+Z(L)\rangle \mid y \in W\setminus L'\} \mid = \frac{q^{k}(q^m-1)}{q-1}.
		\end{equation}
		
		For each $x \in L \setminus W$, by Claim 2, we have $\dim\;\frac{C_L(x) \cap W}{Z(L)} = m$; hence
		\begin{equation}\label{eqse3-3}
			\mid \{ \langle y + Z(L) \rangle \mid y \in W \setminus L', [x,y] = 0\} \mid = \frac{q^m-1}{q-1}.
		\end{equation}
		
		\noindent Hence by \eqref{eqse3-1} and \eqref{eqse3-3}, we have
		\begin{equation}\label{eqse3-4}
			|S|=\frac{q^{m+k}(q^n-1)(q^m-1)}{(q-1)^2},
		\end{equation}
		and the \textbf{proof of Claim 3 is complete}.
		
		\vskip3mm
		
		We now proceed to get the final contradiction. Note that there exists some element $y_0\in W \setminus L'$ such that
		\begin{equation}\label{eqse3-5}
			\mid \{ \langle x+Z(L)\rangle \mid x \in L\setminus W, [x,y_0]=0\} \mid \ge \frac{q^{m}(q^n-1)}{q-1}.
		\end{equation}
		For if there is no such $y_0$ then for each $y \in W \setminus L'$, we get
		$$\mid \{ \langle x+Z(L)\rangle \mid x \in L\setminus W, [x,y]=0\} \mid < \frac{q^{m}(q^n-1)}{q-1}.$$
		But then
		$$|S| < \frac{q^{m+k}(q^n-1)(q^m-1)}{(q-1)^2},$$ which is a contradiction to Claim 3.
		
		\medskip
		
		Note that $L'\le C_L(y_0)$, $\dim\;(C_L(y_0)/L')=m$, $\dim\;(L/W)=n$ and $\dim\;(W/L')=m$.  Assume that
		$$\dim\;\left[\frac{C_L(y_0) + W}{W}\right]=n-s \;\;\text{and} \;\; \dim\;(C_W(y_0)/L')=m-r,$$
		for some non-negative integers $s$ and $r$. Since $\dim\;(L/C_L(y_0))=n$, it follows that $n=s+r$. Then
		$$\dim\;(C_L(y_0)/Z(L))=k+(m-r)+(n-s)=k+m$$
		and
		$$\dim\;(C_W(y_0)/Z(L))=k+(m-r).$$
		Thus,
		\begin{eqnarray*}
			\mid \langle {x+Z(L)} \mid x \in L \setminus W, [x,y_0] = 0\rangle \mid &=& \frac{(q^{k+m}-1)-(q^{k+m-r}-1)}{q-1}\\
			&=& \frac{q^{m+k-r}(q^{n-s}-1)}{q-1}.
		\end{eqnarray*}
		
		The preceding equation along with \eqref{eqse3-5} gives $$\frac{q^m(q^n-1)}{q-1}  \le \frac{q^{m+k-r}(q^{n-s}-1)}{q-1},$$
		which implies that $q^n-1 \leq  q^k - q^{k-r}.$ Hence by Lemma \ref{lem2},
		$$q^n < q^k +1 \le q^{n-1} + 1,$$
		which is absurd as $q\geq 3$ and the proof of the proposition is now complete.
	\end{proof}
	
	As an immediate consequence of the preceding proposition, we get the following important information about centralizers of elements in $L$.
	\begin{corollary}\label{centralizer_info_L}
		Let $L$ be a 3-step nilpotent stem Lie algebra over $\F_q$ of breadth type $(0,n)$ for some $n\geq 1$. Then for all $x \in L \setminus L'$ the following hold:
		\begin{enumerate}
			\item $C_L(x) \cap L'=Z(L).$
			\item $\dim\;(C_L(x)/Z(L)) = \dim\;(L'/Z(L))$.
		\end{enumerate}
	\end{corollary}
	
	\begin{proof}
		For every  $x\in L \setminus Z(L)$ we have $\dim\;(L/C_L(x))=n$. From Proposition \ref{dimension_L_over_L'}, we have $\dim\;(L/L')=n$. Since $L$ is $3$-step nilpotent, the derived subalgebra $L'$ is abelian. Thus we have $C_L(y) = L'$ for all $y\in L'\setminus Z(L)$, and consequently $C_L(x) \cap L'=Z(L)$ for all $x \in L \setminus L'$.
		
		For the second part, observe that given any two elements $x, y \in L\setminus Z(L)$, we have $\dim\;(C_L(x)/Z(L)) = \dim\;(C_L(y)/Z(L))$. Consider $x \in L \setminus L'$ and $y \in L' \setminus Z(L)$. Then we have $\dim\;(C_L(x)/Z(L)) = \dim\;(C_L(y)/Z(L)) = \dim\;(L'/Z(L))$. This completes the proof.
	\end{proof}

	\noindent Now we are ready to prove Theorem~\ref{odd_dimension_three_step_nilpotent}.
	\begin{proof}[\textbf{Proof of Theorem~\ref{odd_dimension_three_step_nilpotent}}]
		It is enough to prove the result for a stem Lie algebra $L$. Suppose $\dim\;(L'/Z(L))=m$. We claim that $n=2m$. Consider $x,y\in L$ with $[x,y]\notin Z(L)$. We consider two cases, namely (1) $n > 2m$ and (2) $n <2m$, and get contradiction in both.
		
		\medskip
		
		\noindent\textbf{\underline{Case I}:} Let $n > 2m$. Write $\overline{L}=L/Z(L)$. Since $\overline{L}$ is of nilpotency class $2$ and dim$(\overline{L}')=m$, we have
		$$\dim\;(\overline{L}/C_{\overline{L}}(\overline{x}))\leq m\;\; \text{and}\;\; \dim\;(\overline{L}/C_{\overline{L}}(\overline{y}))\leq m.$$
		Since $\dim\;(\overline{L}/\overline{L}')=n>2m$, there exists $\overline{w} \notin \overline{L}'$ such that $\overline{w}\in C_{\overline{L}}(\overline{x}) \cap C_{\overline{L}}(\overline{y})$, i.e. $[x,w],[y,w]\in Z(L)$. Since  $w \notin L'$, using Jacobi identity we get $[x,y]\in C_L(w)\cap L'=Z(L)$, a contradiction to our assumption that  $[x,y] \notin Z(L)$.
		
		\medskip
		
		\noindent\textbf{\underline{Case II}:} Let  $n < 2m$. Since $x \not\in L'$, we have
		$$\dim \;\frac{C_L(x)+ L'}{L'} = \dim\;\frac{C_L(x)}{C_L(x)\cap L'}= \dim\;(C_L(x)/Z(L)) = m.$$
		Similarly $\dim\;((C_L(y)+L')/L')=m$.
		Since $n<2m$ and $\dim\;(L/L')=n<2m$, we get $(C_L(x)+ L')\cap (C_L(y)+ L')$ contains $L'$ properly. Consider $w\in (C_L(x)+ L')\cap (C_L(y)+ L')$ with $w\notin L'$. Since $L$ is of nilpotency class $3$, it is easy to see that $[w,x], [w,y]\in Z(L)$. Once again by Jacobi identity $[x,y]\in C_L(w)\cap L'=Z(L)$, a contradiction.
		
		Since we have contradiction in both cases, we conclude that $n=2m$ and the proof is complete.
	\end{proof}

	From the proof of the Theorem~\ref{odd_dimension_three_step_nilpotent}, we get the following:
		
	\begin{corollary}\label{dimension_derived_over_center}
		Let $L$ be a 3-step nilpotent stem Lie algebra of breadth type $(0,2m)$ where $m\geq 1$. Then $\dim\;(L'/Z(L))=m$.
	\end{corollary}

	%%%%%%%%%%%%%%%%%%%%%%%%%%%%%%%%%%%
	
	\section{A 3-step nilpotent Lie algebra of breadth type $(0,n)$}\label{section 3}
	
	In this section we show that for every $m\geq 1$, there exists a finite dimensional 3-step nilpotent Lie algebra of breadth type $(0,2m)$ over a finite field $\F_q$ of characteristic not equal to 2. The construction runs parallel to its $p$-group analogue (see \cite{nky}). 
	
	 We define $\mathfrak{g}_m=\F_{q^m }^5$ which is a vector space of dimension $5m$ over $\F_q$. Now define a bracket $[,]:\mathfrak{g}_m\times \mathfrak{g}_m \to \mathfrak{g}_m$ as follows: for any two elements $\alpha=(a,b,c,d,e)$ and $\beta=(x,y,z,u,v)$ of $\mathfrak{g}_m$ the bracket $[\alpha,\beta]$ is defined by
	 $$[(a,b,c,d,e),(x,y,z,u,v)]=(0,0,ay-bx,ay-bx+2cx-2az,ay-bx+2bz-2cy).$$
	
	\begin{theorem}\label{example_3_step_nilpotent_over_finite_fields}
	$\mathfrak{g}_m$ is a 3-step nilpotent stem Lie algebra of dimension $5m$ over the finite field $\F_q$ where $q$ is odd. Furthermore, $\mathfrak g_m$ is of breadth type $(0,2m)$ where $m\geq 1$.
	\end{theorem}
	
	To prove that $\mathfrak{g}_m$ is a 3-step nilpotent Lie algebra we take an indirect approach. Let $\mathrm{U}_5(\F_{q^m})$ be the Lie algebra (over $\F_q$) of $5\times 5$ strictly upper triangular matrices with entries in the field $\F_{q^m }$, the unique degree-$m$ extension of $\F_q$. We now consider the following subset of $\mathrm{U}_5(\F_{q^m})$:
	\[
	\mathfrak{L}_m=\left\{\begin{bmatrix}
		0 & a & c & d & f \\
		0 & 0 & b & a+b-c & e \\
		0 & 0 & 0 & a & c \\
		0 & 0 & 0 & 0 & b\\
		0 & 0 & 0& b & 0 
	\end{bmatrix}
	\mid a,b,c,d,e,f \in \F_{q^m}
	\right\}
	\]
	For simplicity of notation, we will write $(a,b,c,d,e,f)$ to mean the matrix above. Let $M$ and $N$ be two matrices in $\mathfrak{L}_m$ represented by $(a,b,c,d,e,f)$ and $(x,y,z,u,v,w)$ respectively. Then
	\[[M,N]=MN-NM=\begin{bmatrix} 
		0 & 0 & ay-bx & ay-bx+2cx-2az & av+dy-bu-ex  \\
		0 & 0 & 0 & bx-ay & ay-bx+2bz-2cy\\
		0 & 0 & 0 & 0 &  ay-bx\\
		0 & 0 & 0 & 0 & 0\\
		0 & 0 & 0 & 0 & 0
	\end{bmatrix}.
	\]
	This clearly shows that $\mathfrak{L}_m$ is a Lie subalgebra of $\mathrm U_5(\F_{q^m})$. In 6-tuple notation we have
	\begin{eqnarray*}
		&& \hspace{ 3 cm} [(a,b,c,d,e,f), (x,y,z,u,v,w)]= \\ && (0,0,ay-bx,ay-bx+2cx-2az, ay-bx+2bz-2cy,av+dy-bu-ex).
	\end{eqnarray*}
	It is immediately clear that the center $Z(\mathfrak{L}_m)$ of $\mathfrak{L}_m$ is given by
	$$Z(\mathfrak{L}_m)=\{(0,0,0,0,0,f)\mid f\in \F_{q^m} \}.$$
	We now consider the Lie algebra $\mathfrak{L}_m/Z(\mathfrak{L}_m)$. Observe that the dimension of the Lie  algebra $\mathfrak{L}_m/Z(\mathfrak{L}_m)$ over $\F_q$ is $5m$. An element of $\mathfrak{L}_m/Z(\mathfrak{L}_m)$ is of the form $(a,b,c,d,e,0)+Z(\mathfrak{L}_m)$. Note that this element of $\mathfrak{L}_m/Z(\mathfrak{L}_m)$ can be identified with the element $(a,b,c,d,e)$ of $\mathfrak{g}_m$. With  this identification in mind it is now easy to see that the bracket of $\mathfrak{g}_m$ is nothing but the Lie bracket of  $\mathfrak{L}_m/Z(\mathfrak{L}_m)$. This shows that $\mathfrak{g}_m$ is a Lie algebra over $\F_q$ and $\mathfrak{g}_m\cong \mathfrak{L}_m/Z(\mathfrak{L}_m)$ as $\F_q$-Lie algebras. Since $\mathfrak{L}_m/Z(\mathfrak{L}_m)$ is a nilpotent stem Lie algebra, so is $\mathfrak{g}_m$.
	
	\begin{proposition}\label{Liealgebra_g_m}
		The following properties of $\mathfrak{g}_m$ hold.
		\begin{enumerate}
			\item $\mathfrak{g}_m'=\{(0,0,c,d,e)\mid c,d,e\in \F_q\}$.
			\item $\gamma_{3}(\mathfrak{g}_m)=\{(0,0,0,d,e)\mid d,e\in \F_q\}=Z(\mathfrak{g}_m)$. Consequently, $\mathfrak{g}_m$ is a 3-step nilpotent Lie algebra over $\F_q$.
			\item If $\alpha=(0,0,c,d,e)\in \mathfrak{g}_m'\setminus Z(\mathfrak{g}_m)$, then the centralizer $C_{\mathfrak{g}_m}(\alpha)=\mathfrak{g}_m'$. Thus the breadth of $\alpha$ is $2m$.
			\item If $\alpha\in \mathfrak{g}_m\setminus \mathfrak{g}_m'$, then breadth of $\alpha$ is $2m$.
		\end{enumerate}
	\end{proposition} 
	
	\begin{proof}
		The first two assertions are easy to see.  Suppose $\alpha=(0,0,c,d,e)\in \mathfrak{g}_m'\setminus Z(\mathfrak{g}_m)$. Then $c\neq 0$. If $\beta=(x,y,z,u,v)\in C_{\mathfrak{g}_m}(\alpha)$ we must have $[\alpha,\beta]=0$ which yields $2cx=0$ and $2cy=0$. Since $q$ is odd, we conclude that $x=y=0$. Thus $C_{\mathfrak{g}_m}(\alpha)\subseteq \mathfrak{g}_m'$. Since $\mathfrak{g}'_m$ is abelian we have $C_{\mathfrak{g}_m}(\alpha)=\mathfrak{g}'_m$. 
			
		\medskip
			
		Let $\alpha=(a,b,c,d,e)\in \mathfrak{g}_m\setminus \mathfrak{g}'_m$. We have $(a,b)\neq (0,0)$. If $\beta=(x,y,z,u,v) \in C_{\mathfrak{g}_m}(\alpha)$  then it is easy to see that $bx=ay$, $az=cx$, and $cy=bz$. Suppose $a\neq 0$. Then $y,z$ are uniquely determined. If $b\neq 0$, then $x$ and $z$ are uniquely determined. Thus the dimension of $C_{\mathfrak{g}_m}(\alpha)$ is $3m$, whence breadth of $\alpha$ is $2m$. This completes the proof.
	\end{proof}

	\begin{proof}[\textbf{Proof of Theorem~\ref{example_3_step_nilpotent_over_finite_fields}}]
		The proof follows from the discussions above and Proposition~\ref{Liealgebra_g_m}.
	\end{proof}
	
%%%%%%%%%%%%%%%
	\section{2-step nilpotent Camina Lie algebras arising from finite semifields}\label{section 4}
	
	In the remaining part of this article our aim is to classify the 3-step nilpotent Lie algebras of breadth type $(0,2m)$. In this section we prove a characterization of certain 2-step nilpotent Camina Lie algebras (see Section 4, \cite{kns} for definition) over a finite field $\F_q$ where $q\geq 3$. This characterization will play a central role in determining the structure of the central quotient of a 3-step nilpotent Lie algebra of breadth type $(0,2m)$.
	\begin{theorem}\label{Characterization_Camina_Lie_algebra}
		Let $L$ be a finite dimensional Camina Lie algebra of dimension $3n$, and of nilpotency class $2$ over a finite field $\F_q$ where $q\geq 3$. Let $\dim\;(L/L')=2n$. Then the following statements are equivalent.
		\begin{enumerate}
			\item $L$ is isomorphic to $\rm{U}_3(q^n)$.
			\item All the centralizers of non-central elements of $L$ are abelian.
		\end{enumerate}
	\end{theorem}
	
	The above result is an analogue of a similar result for Camina $p$-groups (see Theorem 3.8, \cite{nky}). The proof in the case of $p$-groups can be found in \cite{ve} and uses the construction of ultraspecial $p$-groups from finite semifields. For a survey of results in semi-extraspecial $p$-groups (equivalently Camina $p$-groups of nilpotency class $2$), ultraspecial $p$-groups and related ideas, we refer the reader to \cite{le}. We take a similar approach to prove Theorem~\ref{Characterization_Camina_Lie_algebra}.
	
	Let $L$ be a 2-step nilpotent Camina Lie algebra over $\F_q$ of dimension $3n$ and $\dim\;(L/L')=2n$. Further assume that $L$ has at least two abelian subalgebras of dimension $2n$ over $\F_q$. We show that such a Lie algebra $L$ can always be obtained from a finite semifield.
%%%%%
	\subsection{Finite Semifields and Semifield Lie algebras over finite fields}
	We first recall the definition of a semifield. Let $F$ be a non-empty set with two binary operations $+$ and $.$ such that $(F,+)$ is an abelian group. We call $F$ a pre-semifield if the following axioms hold:
	\begin{align*}
	&\hspace{-2.5 cm} a.(b+c)=a.b+a.c \hspace{1 cm}\text{for every $a,b,c\in F$}. \tag{F1}\\
	&\hspace{-2.5 cm}(b+c).a=b.a+b.c \hspace{1 cm} \text{for every $a,b,c\in F$}. \tag{F2}\\
	&\hspace{-2.5 cm} a.c=0 \implies a=0 \text{ or } c=0 \hspace{1 cm} \text{for every $a,c\in F$}. \tag{F3}
	\end{align*}
	Further, $F$ is called a semifield if there exists a multiplicative identity. Note that (F3) does not imply the existence of inverse. If $F$ is a finite semifield, then (F3) implies that every non-zero element has a left and right inverse. Thus $F$ is a (possibly non-associative) division ring.  In general, $F$ is a vector space over the prime field $\mathbb{F}_p$ consisting of $p$ elements, where $p$ is a prime. Thus $|F|$ is a power of a prime $p$. Further if $'.'$ is associative, by Weddernburn's theorem $F$ becomes a finite field. From now on we will assume $F$ is finite. Thus for us $F$ is a (possibly non-associative) division algebra over a finite field $\F_q$. For simplicity, we will say that $F$ is a semifield over the field $\F_q$. If the operation $'.'$ is non-associative, $F$ is called a proper semifield.
	
	\medskip
	
	Let $F$ be a finite semifield over a finite field $\F_q$. Thus $|F|=q^n$ for some $n\geq 1$. We now construct a Lie algebra of dimension $3n$ over the finite field $\F_q$ using $F$. Let $L(F):=\{(a,b,c) \mid a,b,c\in F\}=F\times F\times F$. The usual component-wise addition and scalar multiplication makes $L(F)$ a vector space of dimension $3n$ over $\F_q$. Define a bracket operation $[\ ,\ ]:L(F)\times L(F)\to L(F)$ by $[(a_1,b_1,c_1),(a_2,b_2,c_2)]=(0,0,a_1b_2-a_2b_1).$ It can be easily verified that $L(F)$ with the above bracket operation becomes a Lie algebra over $\F_q$. It is easy to see that $L(F)'=Z(L(F))=\{(0,0,a)\mid a\in F\}$ and further for every $x\in L(F)\setminus L(F)'$, $[x,L(F)]=L(F)'$. Thus $L(F)$ is a 2-step nilpotent Camina Lie algebra over $\F_q$. We have $\dim\;(L(F)/L(F)')=2n$. Consider the following two Lie subalgebras of $L(F)$.
	$$A^*=\{(a,0,c)\mid a,c\in F\} \;\;\;\text{and} \;\;\; B^*=\{(0,b,c)\mid b,c\in F\}.$$
	Clearly, $A^*$ and $B^*$ are both abelian Lie subalgebras of $L(F)$. Thus from a finite semifield $F$ over the finite field $\F_q$ such that $|F|=q^n$, we have constructed a 2-step nilpotent Camina Lie algebra $L(F)$ of dimension $3n$ over $\F_q$ with $\dim\;(L(F)/L(F)')=2n$. Further it has at least two abelian Lie subalgebra of dimension $2n$ over $\F_q$. We will call this the semifield Lie algebra associated to the semifield $F$. Note that if $F$ becomes a field, that is, $F\cong \F_{q^n}$ then $L(F)\cong \mathrm{U}_3(q^n)$. 
	
	\medskip
	
	There are many constructions of semifields over $\F_q$ (see \cite{ka, kn, cw}). The following example is of the classical Dickson's commutative semifields over $\F_q$.
	\begin{example}\label{Dickson}
		Let $q=p^n$ where $n>1$ and $p$ is an odd prime. Let $F$ be the two-dimensional vector space $F=\F_q\times \F_q$ over $\F_q$. Let $k$ be a non-square in $\F_q$. Define multiplication $*$ on $F$ as follows: for $(a,b)$ and $(c,d)$ in $F$,
		$$(a,b)*(c,d)=(ac+k\sigma(b)\sigma(d), ad+bc)$$
		where $\sigma$ is a non-identity field automorphism of $\F_q$. It can be easily checked that  $F$ is a commutative semifield over $\F_q$. In fact $*$ is always non-associative and hence $F$ is a proper semifield over $\F_q$.
	\end{example}

%%%%%
\subsection{2-step nilpotent Camina Lie algebra}

	The following recognition theorem is a major step towards the proof of Theorem~\ref{Characterization_Camina_Lie_algebra}.
	
	\begin{theorem}\label{Camina_recognition_semifield_Lie_algebra}
		Let $L$ be a 2-step nilpotent Camina Lie algebra of dimension $3n$ over a finite field $\F_q$. Suppose $\dim\;(L/L')=2n$ and there is at least two abelian Lie subalgebras of $L$ of dimension $2n$ over $\F_q$. Then there exists a finite semifield $F$ over $\F_q$ with $|F|=q^n$ such that $L\cong L(F)$.
	\end{theorem}
	
	\begin{proof}
		Suppose $A^*$ and $B^*$ are two abelian Lie subalgebras of $L$ of dimension $2n$ over $\F_q$. Since $L$ is a 2-step nilpotent Camina Lie algebra we must have $A^*\cap B^*=L'=Z(L)$. Thus $A^*=A\oplus L'$ and $B^{*}=B\oplus L'$ for some Lie subalgebra $A$ and $B$ of $L$ of dimension $n$ over $\F_q$. Thus $A\cong (\F_q)^n$ as $\F_q$-vector space. Same holds true for $B$ as well as $L'$. Further we also get that $L=A+B+L'$. Let $x,y\in L=A+B+L'$. Then $x=a_1+b_1+c_1$ and $y=a_2+b_2+c_2$ and $$[x,y]=[a_1+b_1+c_1,a_2+b_2+c_2]=[a_1,b_2]+[b_1,a_2]=[a_1,b_2]-[a_2,b_1].$$
		Consider the set $F=(\F_q)^n$ which is the $n$-dimensional vector space over $\F_q$. We give $F$ a semifield structure by defining multiplication $*:F\times F\to F$ as $a_1*b_2=[a_1,b_2]$ for all $a_1\in F,\; b_2\in F$. Here we consider the first copy of $F$ as $A$, the second one as $B$ and the product then goes to $L'$ whose underlying set is also $F$. One can check that $F$ becomes a semifield of order $q^n$. The semifield Lie algebra $L(F)$ associated to $F$ has the bracket operation $[(a_1,b_1,c_1),(a_2,b_2,c_2)]=a_1b_2-a_2b_1=[a_1,b_2]-[a_2,b_1]$. This immediately shows that $L\cong L(F)$ as $\F_q$-Lie algebras.
	\end{proof}
	We now recall the definition of isotopic semifields (see Section 9, \cite{le}). We say that two semifields $(F_1,+,*_1)$ and $(F_2,+,*_2)$ are \textbf{isotopic} if there exists additive isomorphisms $\alpha,\beta,\gamma:F_1\to F_2$ such that $\gamma(a*_1b)=\alpha(a)*_2\beta(b)$ for all $a,b\in F_1$. Isotopy of semifields is an equivalence relation. The next theorem says that isotopic semifields give rise to isomorphic semifield Lie algebras. 	We omit the proof since it is straight-forward. 
	
	\begin{theorem}\label{isotopic_semifields_isomorphic_Lie_algebra}
		If $F_1$ and $F_2$ are two finite isotopic semifields  then $L(F_1)$ is isomorphic to $L(F_2)$.
	\end{theorem}
	
	\begin{remark}
		Note that any pre-semifield $P$ is isotopic to a semifield $F$. Further if $P$ is commutative so is $F$. The previous result shows that it suffices to work with a semifield in our case.
	\end{remark}

	Let us now consider the centralizer of a non-central element in the semifield Lie algebra $L(F)$ over $\F_q$ associated to a semifield $F$ (over $\F_q$) with $|F|=q^n$. Let $a_1,b_1\in F$ such that $a_1b_1\neq 0$. For an element $t:=(a_1,b_1,c_1)\in L(F)$, it is clearly seen that the centralizer $C_{L(F)}(t)$ is given by $C_{L(F)}(t)=\{(a_2,b_2,c_2)\mid a_1b_2=a_2b_1\}$. Following \cite{ks}, for $(a_1,b_1)\in F\times F$, we define $C_{(a_1,b_1)}:=\{(a_2,b_2)\in F\times F\mid a_1b_2=a_2b_1\}$. It follows that $C_{L(F)}(a_1,b_1,c_1)=C_{(a_1,b_1)}\times F$. 
	We will call $C_{(a_1,b_1)}$ abelian if $C_{a_1,b_1}\times F$ is abelian. For a commutative semifield $F$, we now define $\text{Mid}(F)=\{z\in F\mid x.(z.y)=(x.z).y \; \text{for every}\; x,y\in F\}.$ The set $\text{Mid}(F)$ is a finite field and $|\text{Mid}(F)|=q^h$ where $h\leq n$. $\text{Mid}(F)$ is called the middle nucleus of $F$. The following lemma is the final ingredient towards the proof of Theorem~\ref{Characterization_Camina_Lie_algebra}.
	
	\begin{lemma}[Lemma 4.3, \cite{ks}]\label{commutative_isotopism} Let $F$ be a finite semifield. Let $(a_1,b_1)\in F^2$ be such that  $a_1b_1\neq 0$ and $C_{(a_1,b_1)}$ is abelian. Let $x$ be the solution of $xa_1=b_1$. Then,
		\begin{enumerate}
			\item We have $C_{(a_1,b_1)}=C_{(1,x)}=\{(r,rx) \mid r\in F\}.$
			
			\item The semifield $F$ is isotopic to a commutative one.
			
			\item If $F$ is commutative then $x\in \rm{Mid}(F)$.
		\end{enumerate}
		Conversely, if $F$ is commutative then $C_{(1,x)}$ is abelian for all $x\in \rm{Mid}(F)$. 
	\end{lemma}

	\begin{proof}[\bf Proof of Theorem~\ref{Characterization_Camina_Lie_algebra}]
		Using direct computations with matrices it is easy to verify that $L=\rm{U}_3(q^n)$ satisfies the theorem. Conversely, let $L$ be a 2-step nilpotent Camina Lie algebra of dimension $3n$ over $\F_q$ with $\dim\;(L/L')=2n$. Assume that the centralizer of every non-central element in $L$ is abelian. Since $L$ is Camina and $\dim\;(L/L')=2n$ we have $\dim\;(C_L(x))=2n$ for all $x\in L\setminus L'$. Considering any two such distinct centralizers, we can say that $L$ has at least two abelian Lie subalgebras of dimension $2n$ over $\F_q$.  Using Theorem~\ref{Camina_recognition_semifield_Lie_algebra}, we conclude that $L$ is isomorphic to a semifield Lie algebra $L(F)$ associated to some semifield $F$ (over $\F_q$) of order $q^n$. By the previous lemma and using Theorem~\ref{isotopic_semifields_isomorphic_Lie_algebra}, we can assume $F$ to be a commutative semifield. Once again by Lemma~\ref{commutative_isotopism}, and the fact that every centralizer is abelian it immediately follows that $\text{Mid}(F)=F$. Thus $F$ is a field of order $q^n$, whence it follows that $L\cong \rm{U}_3(q^n)$. 
	\end{proof}
	
	\begin{example}
		Consider the semifield $F$ of Example~\ref{Dickson} where the automorphism $\sigma$ is the Frobenius automorphism $\sigma:\F_q\to \F_q$ defined by $x\mapsto x^p$. $F$ is a proper semifield over $\F_q$. Note that $\alpha=(1,0)\in F$ is the multiplicative identity. Let $\0=(0,0)$ be the additive identity of $F$. Let $L:=L(F)$ be the semifield Lie algebra over $\F_q$ corresponding to $F$. Note that $\dim L(F)=6$. Further $Z(L)=L'=\{(\0,\0,(c_1,c_2))\mid c_1,c_2\in \F_q\}$. Let $\Gamma=(\alpha, \alpha',\0)$ where $\alpha'=(0,1)\in F$. Thus $\Gamma\in L\setminus Z(L)$. The centralizer $C_{L}(\Gamma)$ is 
		$$C_L(\Gamma)=\{((x_1,x_2),(kx_2^p,x_1),(c_1,c_2))\mid x_1,x_2,c_1,c_2\in \F_q \}.$$
		We show that $C_L(\Gamma)$ is non-abelian. Take $A=((x_1,x_2),(kx_2^p,x_1),(u_1,u_2)) \in L$ and $B=((y_1,y_2),(ky_2^p,y_1),(v_1,v_2))\in L$. Then $[A,B]=(\0,\0,\beta)$ where $\beta=(x_1,x_2)*(ky_2^p,y_1)-(kx_2^p,x_1)*(y_1,y_2)$. We get
		$$\beta=(k(x_1y_2^p+x_2^py_1^p-x_2^py_1-x_1^py_2^p), k(x_2y_2^p-x_2^py_2)).$$
		Let $x_1=1,x_2=1$. Then $\beta=(y_1^p-y_1,y_2^p-y_2)$. If $\beta=\0$ then we must have $y_1,y_2\in \F_p$. Since $q=p^n$ where $n>1$, it follows that $C_L(\Gamma)$ is non-abelian. This is consistent with the fact that $\text{Mid}(F)=\F_q\neq F$.
	\end{example}
	
%%%%%%%%
\subsection{3-step nilpotent stem Lie algebras}

	As a consequence of Theorem~\ref{Characterization_Camina_Lie_algebra}, we get the following result.
	
	\begin{theorem}\label{central_quotient_recognition}
		Let $L$ be a 3-step nilpotent stem Lie algebra over $\F_q$ of breadth type $(0,2m)$ where $m\geq 1$. Then $L/Z(L)$ is isomorphic to $\rm{U}_3(q^m)$.
	\end{theorem}
	
	\begin{proof}
		It is enough to prove that $L/Z(L)$ satisfies the hypothesis of Theorem~\ref{Characterization_Camina_Lie_algebra}. Let $L/Z(L)= \overline{L}$. By Proposition~ \ref{dimension_L_over_L'} and Corollary~\ref{dimension_derived_over_center}, $\dim\;(L/L')=2m$ and
		$\dim\;(L'/Z(L))=m$. Hence
		$$\dim\;(\overline{L})=3m~~ \mbox{and}~~\dim\;(\overline{L}')=m.$$
		Thus $\dim\;(\overline{L}/\overline{L}')=2m$. Since $L'$ is abelian we must have $C_L(x)=L'$ for every $x\in L'\setminus Z(L)$.
		
		\medskip
		
		For $x\in L$, let $\overline{x}=x+Z(L)$. We now claim that $\dim\;(C_{\overline{L}}(\overline{x}))=2m$ for all $\overline{x}\in \overline{L}\setminus \overline{L}'$, from which it will follow that $\overline{L}$ is a Camina Lie algebra.
		
		\medskip
		
		Since $\overline{L}$ is of nilpotency class $2$, $\overline{L}'$ is contained in the centralizer of every element  in $\overline{L}$. Hence $\overline{C_L(x)+L'}\leq C_{\overline{L}}(\overline{x})$ for every $x\in L\setminus L'$. Since $\dim\;(C_L(x)/Z(L))=\dim\;(L'/Z(L))=m$ and $C_L(x)\cap L'=Z(L)$ (see Lemma~\ref{centralizer_info_L}), we get $\dim\;(\frac{C_L(x)+L'}{Z(L)})=2m$ and $\dim\;(\frac{L}{C_L(x)+L'})=m$. Hence $\dim\;(\overline{L}/C_{\overline{L}}(\overline{x}))\leq m$ and $\dim\;(C_{\overline{L}}(\overline{x}))\geq 2m$.
		
		\medskip
		
		Fix $x\in L\setminus L'$. If possible, suppose that
		$\dim\;(C_{\overline{L}}(\overline{x})) > 2m$, that is,
		$\dim\;(\overline{L}/C_{\overline{L}}(\overline{x}))<m$. We claim that
		$\overline{x}\notin Z(\overline{L})$. If not, let $\overline{x}\in
		Z(\overline{L})$. For any minimal generating set $\{x_1,\ldots, x_{2m}\}$ of $L$, $\{\overline{x}_1,\ldots,\overline{x}_{2m}\}$ is also a minimal generating set for $\overline{L}$. Then $[\overline{x},\overline{x}_i]=0$, that is, $[x,x_i]\in Z(L)$ for $1\le i\le 2m$. Then using Jacobi identity, for $1\le i\le 2m$, $[x_i,x_j]\in C_L(x)\cap L'=Z(L)$ (since $x\notin L'$); hence $L'\subseteq Z(L)$, a contradiction. Thus there exists $t\in L$ such that $[t,x]\notin Z(L)$. 
		
		Since $\dim\;(\overline{L}')=m$, we have $\dim\;(\overline{L}/C_{\overline{L}}(\overline{t})) \leq m$, and therefore it follows that $C_{\overline{L}}(\overline{x})\cap C_{\overline{L}}(\overline{t})$ contains $\overline{L}'$ properly. Take $\overline{w}\in \left(C_{\overline{L}}(\overline{x})\cap C_{\overline{L}}(\overline{t})\right)\setminus \overline{L}'$. Then $[w,x], [w,t]\in Z(L)$ with $w \notin L'$. Once again using Jacobi identity we can conclude that $[x,t]\in C_L(w) \cap L'=Z(L)$, a contradiction. Thus $\dim\;(C_{\overline{L}}(\overline{x}))=2m=\dim\;(\overline{C_L(x)+L'})$ and our claim is established. As a consequence we also get that $\overline{C_L(x)+L'}= C_{\overline{L}}(\overline{x}).$ 
		
		\medskip
		
		Finally to finish the proof we need to show that centralizers of non-central elements of $\overline{L}$ are abelian. Consider any $x\in L\setminus L'$. For $y_1,y_2\in C_L(x)$ we must have $[y_1,y_2]\in C_L(x)\cap L'=Z(L)$. Hence
		$[C_L(x),C_L(x)] \leq Z(L)$. Since $L$ is of nilpotency class $3$, we have $$[C_L(x)+L', C_L(x)+L'] = [C_L(x),C_L(x)]+ [C_L(x),L'] +[L',L']\leq Z(L),$$
		i.e. $C_{\overline{L}}(\overline{x})=\overline{C_L(x)+L'}$ is abelian. 
	\end{proof}
	
	\noindent The following is a consequence of the proof of the above theorem.
	
	\begin{corollary}\label{se4-cor1}
		Let $L$ be a 3-step nilpotent stem Lie algebra over $\F_q$ of breadth type $(0,2m)$ and  $u, v \in L \setminus L'$ be such that $[u,v]
		\in Z(L)$. Then there exists an element $h \in L'$ such that $[u,v+h] = 0$. In particular, $[u,v]\in\gamma_3(L)$. 
	\end{corollary}
	
	\begin{proof}
		We have $[u+Z(L),v+Z(L)]=0$ which implies $v+Z(L)\in C_{\overline{L}}(u+Z(L))$. From the proof of the previous theorem we have $v\in C_{L}(u)+L'$. The proof now follows.
	\end{proof}
	
	For a Lie algebra $L$, the following is always true by virtue of the Jacobi identity.
	\begin{equation}\label{eqse3-5a}
		\mbox{ If } [a,b]\in Z(L) \mbox{ then } [[a,t],b]=[[b,t],a]
		\hskip5mm (a,b,t \in L).
	\end{equation}
	
	\noindent We conclude this section with the following information on the center of $L$.
	
	\begin{proposition}\label{dimension_center}
		Let $L$ be a 3-step nilpotent stem Lie algebra of breadth type $(0,2m)$. Then $Z(L)=\gamma_3(L)$ and $\dim\;(Z(L))=2m$.
	\end{proposition}
	
	\begin{proof}
		Consider $x_1\in L\setminus L'$. From Corollary~\ref{dimension_derived_over_center} and  Proposition~\ref{dimension_L_over_L'}, we have that $C_L(x_1)\cap L'=Z(L)$, $\dim\;(C_L(x_1)/Z(L)) = \dim\;(L'/Z(L))=m$ and $\dim\;(L/L')=2m$.  Let $C_L(x_1)=\langle x_1,\ldots,x_m,Z(L)\rangle$. Consider $y_1\in L\setminus C_L(x_1)+L'$. Let $C_L(y_1)=\langle y_1,\ldots, y_m,Z(L)\rangle$.
		Then $C_{\overline{L}}(\overline{x}_1)=\langle \overline{x}_1,\ldots,\overline{x}_m,\overline{L}'\rangle$ and
		$C_{\overline{L}}(\overline{y}_1)=\langle
		\overline{y}_1,\ldots,\overline{y}_m,\overline{L}'\rangle$ are distinct proper centralizers of $\overline{L}$. Using Theorem~\ref{central_quotient_recognition}, they generate $\overline{L}$. It follows that $\{ x_1,\ldots, x_m,y_1,\ldots,y_m\}$ is a (minimal) generating set for $L$.
		Define
		\begin{align*}
			[x_1,y_i] =h_i
			\hskip5mm (1\le i \le m).
		\end{align*}
		From Theorem~\ref{central_quotient_recognition}, we get that $h_1,\ldots, h_m$ are independent modulo
		$Z(L)$, and  $$L'=\langle h_1,\ldots, h_m,Z(L)\rangle.$$
		
		Now we define
		$$[h_1,x_i]=z_i,  \mbox{ and }  [h_1,y_i]=z_{m+i} \hskip5mm (1 \le i \le m).$$ Since $x_1, \dots, x_m, y_1, \dots, y_m$ are independent modulo $L' = C_L(h_1)$, it follows that
		$$[h_1,x_1],\ldots, [h_1,x_m], [h_1,y_1],\ldots, [h_1,y_m]$$
		are independent and they generate a subalgebra  $K$ of dimension 2m in $\gamma_3(L)$.
		
		\medskip
		
		We now proceed to show that $K=\gamma_3(L)$. It is sufficient to show that for any $h\in L'\setminus Z(L)$, $[h,x_i],[h,y_i]\in K$ for $1 \le i \le m$. For any $h\in L'\setminus Z(L)$ and a fixed $i$ with $1 \le i \le m$, consider $[h,x_i]$. Let $A=\langle x_1,\ldots, x_m,Z(L)\rangle$ and $B=\langle y_1,\ldots,y_m,Z(L)\rangle$. Since $\langle [x_1,y_1], \dots, [x_1,y_m],  Z(L) \rangle=L'$, there exists $y\in B$ such that $h =[x_1,y]+Z(L)$. Then by \eqref{eqse3-5a},
		\begin{align*}
			[h,x_i] &=[[x_1,y],x_i]=[[x_i,y],x_1].
		\end{align*}
		Again since $\langle [x_1,y_1],\dots, [x_m,y_1], Z(L) \rangle=L'$, there exists $x\in A$ such that $[x_i,y]=[x,y_1]+Z(L)$. Once again by
		\eqref{eqse3-5a}, we get
		\begin{align*}
			[h,x_i]=[[x_i,y],x_1] =[[x,y_1],x_1]&=[[x_1,y_1],x]=[h_1,x]\in K.
		\end{align*}
		Similarly we can show that $[h,y_i]\in K$, $1\le i\le m$; hence $K=\gamma_3(L)$ and is of dimension 2m over $\F_q$.
		
		\medskip
		
		Finally we show that  $Z(L)=\gamma_3(L)$. For this, since
		$[x_1,y_1],\ldots, [x_1,y_m]$ are independent modulo $Z(L)$ and
		$$L'=\langle [x_1,y_1], \ldots, [x_1,y_m], Z(L)\rangle,$$
		it suffices to prove that
		\begin{equation}\label{eqse3-6}
			L'=\langle [x_1,y_1], \ldots, [x_1,y_m], \gamma_3(L)\rangle.
		\end{equation}
		Note that $\gamma_3(L)\subseteq Z(L)$ and $\gamma_3(L)=\langle z_1,\ldots, z_{2m}\rangle$. Also
		$$A=\langle x_1,\ldots, x_m,Z(L)\rangle=C_L(x_1) \mbox{ and } B=\langle y_1,\ldots,y_m,Z(L)\rangle=C_L(y_1).$$
		
		If \eqref{eqse3-6} does not hold then there exist $z\in Z(L)\setminus \gamma_3(L)$ and a Lie bracket $[x_i, y_j]$ for some $i,j$ with $1\leq i,j\leq m$,  such that
		\begin{equation*}
			[x_i,y_j]=e_1[x_1,y_1]+\cdots+e_m[x_1,y_m]+z,
		\end{equation*}
		where $e_i \in \mathbb{F}_q$ for $1 \le i \le m$.
		Let $y=\sum\limits_{i=1}^{m}e_iy_i$. Then the preceding equation implies
		$[x_i,y_j]-[x_1,y]-z=0$. Now since $[x_i,x_1]=0$, we have
		$$[x_i+y, x_1+y_j]=[x_i,y_j]+[y,x_1]+[y,y_j]= z+[y,y_j].$$ 
		Since $[y,y_j]\in Z(L)$, this implies $[x_i+y,x_1+y_j]\in Z(L)$. By Corollary \ref{se4-cor1}, both  $[y,y_j]$ and $ [x_i+y,x_1+y_j]$  both belong to $\gamma_3(L)$. Consequently $z\in \gamma_3(L)$, a contradiction. 
		
		\noindent This proves that \eqref{eqse3-6} holds. Hence
		$\gamma_3(L)=Z(L)$ and the proof is complete.
	\end{proof}
	
	%%%%%%%%%%%%%%%%%%%%%%%%%%%%%%%%%%%%%%%%%

	\section{Proof of Theorem~\ref{even_dimension_three_step_nilpotent}}\label{section 5}
		Let $L$ be a 3-step nilpotent stem Lie algebra over $\F_q$ ($q\geq 3$) of breadth type $(0,2m)$. We have shown that $L/Z(L)$ is isomorphic to the Lie algebra ${\rm U}_3(q^m)$ over $\F_q$ (see Theorem~\ref{central_quotient_recognition}). Our strategy for proving Theorem~\ref{even_dimension_three_step_nilpotent} is to obtain a presentation of the Lie algebra $L$ from a presentation of ${\rm U}_3(q^m)$. Until the proof of Theorem~\ref{even_dimension_three_step_nilpotent}, we assume that $m\geq 2$. Once again we assume $q\geq 3$ throughout unless otherwise mentioned.
		
		\medskip
		
		At first we obtain a set of structure constants of ${\rm U}_3(q^m)$. Recall that $\mathbb{F}_{q^m}$ denotes the finite field of order $q^m$. Then $\mathbb{F}_{q^m}=\mathbb{F}_q(\alpha)$, where $\alpha$ satisfies a monic irreducible polynomial of degree $m$ over $\mathbb{F}_q$. Consider the following matrices in ${\rm U}_3(q^m)$ for any integer $i \ge 1$:
		$$X_i=
		\begin{bmatrix}
		0 &  \alpha^{i-1} & 0\\
		  &  0            & 0 \\
		  &               & 0
		\end{bmatrix},
		\hskip5mm
		Y_i=
		\begin{bmatrix}
		0 &  0 & 0\\
		  &  0 & \alpha^{i-1} \\
		  &    & 0
		\end{bmatrix},
		\hskip5mm
		H_i=
		\begin{bmatrix}
		0 &  0 & \alpha^{i-1} \\
		&  0 & 0 \\
		&    & 0
		\end{bmatrix},
		$$
	Then it is easy to see that $\Theta = \{X_1,\ldots, X_m, Y_1,\ldots, Y_m\}$ is a minimal generating set for ${\rm U}_3(q^m)$ and $\{H_1,\ldots, H_m\}$ is a minimal generating set for the center (as well as the derived subalgebra) of ${\rm U}_3(q^m)$. Further, these matrices satisfy the following relations.
	\begin{align}
	& [X_i,X_j]=[Y_i,Y_j]=0,\label{eqse5-1}\\
	& [H_i,X_j]=[H_i,Y_j]=0,\label{eqse5-2}\\
	& [X_i,Y_j]=H_{i+j-1} \mbox{ for all }i,j\geq 1.\label{eqse5-3}
	\end{align}

	Since $\mathbb{F}_q(\alpha)$ is a vector space over $\mathbb{F}_q$ with basis $\{1,\alpha,\ldots, \alpha^{m-1}\}$, for $\alpha^{i+j-2} \in\mathbb{F}_q(\alpha)$, there exist unique $\kappa_{i,j,1}, \ldots, \kappa_{i,j,m}\in\mathbb{F}_q$ such that
	$$\alpha^{i+j-2} = \kappa_{i,j,1} + \kappa_{i,j,2}\,\alpha + \cdots +
	\kappa_{i,j,m}\,\alpha^{m-1}.$$
	Then by \eqref{eqse5-3}, we have
	$$H_{i+j-1} = \sum_{t=1}^{m} \kappa_{i,j,t}H_t = \left[X_1,\sum_{t=1}^{m}\kappa_{i,j,t}Y_t\right],$$
	which in turn  implies that
	\begin{align}\label{eqse5-4}
	[X_i,Y_j]=[X_1,\kappa_{i,j,1}Y_1+\kappa_{i,j,2}Y_2+\cdots +
	\kappa_{i,j,m}Y_m] \hskip5mm (1\leq i,j\leq m).
	\end{align}
	The constants $\kappa_{i,j,l}$ for $1\leq i,j,l\leq m$ are the structure constants of ${\rm U}_3(q^m)$ with respect to the set of minimal generators given by $\Theta$. Note that the set $\Theta$ depends on the field generator $\alpha$ of $\F_{q^m}$ over $\F_q$. Thus the structure constants (which depends on $\Theta$) depends on the choice $\alpha$.
	
	\medskip
	
	For $1\leq i,j,l\leq m$, the generators $X_i,Y_i,H_i$, the constants $\kappa_{i,j,l}$, and the relations \eqref{eqse5-1} - \eqref{eqse5-3} give a presentation of the Lie algebra ${\rm U}_3(q^m)$. It is clearly seen that
	\begin{equation}\label{eqse5-6}
		\kappa_{i,j,l}=\kappa_{j,i,l}.
	\end{equation}
 and also that  $[X_i,Y_j]=[X_j,Y_i]$.
	From \eqref{eqse5-4} with the relation $[X_i,Y_j]=[X_j,Y_i]$,
	we obtain
	\begin{align}\label{eqse5-7}
	[X_j,Y_i]=[\kappa_{i,j,1}X_1 + \kappa_{i,j,2}X_2+ \cdots
	+\kappa_{i,j,m}X_m,Y_1] \hskip5mm (1\leq i,j\leq m).
	\end{align}

	We now construct a presentation of a stem Lie algebra $L$ of breadth type $(0,2m)$.

	\begin{lemma}\label{presentation_class3_(0,m)_Liealgebra}
		Let $L$ be a 3-step nilpotent stem Lie algebra over the finite field $\F_q$ of breadth type $(0,2m)$. Then there exist
		$$\alpha_{i,j,l}, ~~ \beta_{i,j,l}, ~~ \gamma_{i,j,l}, ~~ \delta_{i,j,l}, ~~
		\lambda_{i,j,l}\in
		\F_q \hskip5mm (1\leq i, j \leq m, \,\,1\leq l \le 2m),$$
		such that $L$ admits the following presentation:
		\begin{align*}
		L=\Big{\langle}  &x_1,\ldots, x_m,y_1,\ldots, y_m, h_1, h_2,\ldots, h_m, z_1,
		z_2, \ldots, z_{2m}~~ \Big{|} \\
		& [z_k,z_r]=[z_k,x_i]=[z_k,y_i]=[z_k,h_i]=0 \hskip3mm (1\leq k,r,\leq 2m,
		1\leq i\leq m), \tag{R1}\\
		& [h_i,h_j]=0 \hskip5mm (1\leq i,j\leq m),\tag{R2}\\
		& [h_i, x_j]=\sum_{l=1}^{2m} \gamma_{i,j,l}z_l \hskip5mm (1\leq i,j\leq
		m), \tag{R3}\\
		& [h_i, y_j]=\sum_{l=1}^{2m}\delta_{i,j,l}z_l \hskip5mm (1\leq i,j\leq
		m), \tag{R4}\\
		& [x_i,x_j]=\sum_{l=1}^{2m} \alpha_{i,j,l}z_l\hskip5mm (1\leq i,j\leq m),
		\tag{R5}\\
		& [y_i,y_j]=\sum_{l=1}^{2m} \beta_{i,j,l}z_l\hskip5mm (1\leq i,j\leq m),
		\tag{R6}\\
		& [x_1, y_i] = h_i, ~~[h_1, x_i] = z_i, ~~[h_1, y_i] = z_{m+i} \hskip3mm
		(1\leq i\leq m) \tag{R7},\\
		& [x_i,y_j]=\sum_{t=1}^{m}\kappa_{i,j,t}h_t+ \sum_{l=1}^{2m} \lambda_{i,j,l}z_l  \hskip5mm (1\leq
		i,j\leq m)\tag{R8} \Big{\rangle}. 
		%& [x_j,y_i]=\left[\sum_{t=1}^{m}\kappa_{i,j,t}x_t,y_1\right]+\sum_{l=1}^{2m} \mu_{i,j,l}z_l
		%\hskip5mm (1\leq i,j\leq m), \tag{R8}\\
		\end{align*}
		where   $\kappa_{i,j,l}$,  $1 \le i, j, l \le m$, are the structure constants of ${\rm U}_3(q^m)$.
	\end{lemma}

	\begin{proof}
	By Proposition~\ref{dimension_L_over_L'}, Corollary~\ref{dimension_derived_over_center} and Proposition~\ref{dimension_center} we have
	$$\dim\;(L/L')=2m, ~~ \dim\;(L'/Z(L))=m, ~~ \dim\;(Z(L))=2m.$$
	We will obtain the required presentation of $L$ from the presentation of $L/Z(L) \cong {\rm U}_3(q^m)$.
	
	\medskip
	
	Since $\dim\;(Z(L))=2m$, $Z(L)$ is minimally spanned by $2m$ elements  say $z_1,\ldots, z_{2m}$. Let $\varphi:L\rightarrow {\rm U}_3(q^m)$ denote a surjective Lie-algebra homomorphism with $\ker\varphi=Z(L)$. Recall that $X_i$'s, $Y_i$'s and $H_i$'s are  the generators of ${\rm U}_3(q^m)$. Choose $ x_1,\ldots, x_m, y_1,\ldots, y_m$, $h_1,\ldots, h_m$
	in $L$ such that
	$$\varphi(x_i)=X_i, \hskip3mm \varphi(y_i)=Y_i\hskip3mm \mbox{ and } \hskip3mm \varphi(h_i)=H_i\hskip5mm (1\leq i\leq m).$$
	It is clear that
	$$\{x_1,\ldots, x_m,y_1,\ldots, y_m, h_1, h_2,\ldots, h_m, z_1,
	z_2, \ldots, z_{2m}\}$$
	spans $L$, and the set $$\{h_1, h_2,\ldots, h_m, z_1, z_2, \ldots, z_{2m}\}$$ spans $L'$.
	
	\medskip
	
	The relations in (R1) are obvious. Since $L'$ is abelian, we get $[h_i,h_j]=0$ for every $1\le i,j\le m$. This gives the relations in (R2). The relations
	\eqref{eqse5-1} and \eqref{eqse5-3} give the relations (R3)-(R6) for some $\alpha$'s, $\beta$'s, $\gamma$'s, $\delta$'s in $\mathbb{F}_q$. Further, the relations (R8) can be easily obtained from the fact that
	$$[X_i,Y_j]=\sum_{t=1}^{m} \kappa_{i,j,t}H_{t} \;\;\;\;\;\; (1\leq i,j\leq m).$$
	
	Finally we obtain the relation in (R7). Since, for $1 \le i \le m$,  $[X_1, Y_i] = H_i$ we have $[x_1, y_i] - h_i \in Z(L).$ Thus $[x_1, y_i] = h_i+w_i$ for some $w_i \in Z(L)$. Replacing $h_i$ by $h_i+w_i$, which do not violate any of the preceding relations we obtain
	\begin{equation}\label{lemma 5.1_eqn1}
		[x_1, y_i] = h_i \hskip5mm (1 \le i \le m).
	\end{equation}
	Finally, we have $h_1\in L'\setminus Z(L)$. Further, $\dim\;(L/L')=2m$ and $L'$ is abelian, hence $C_L(h_1)=L'$. Since $x_1,\ldots, x_m, y_1,\ldots, y_m$ are independent modulo $C_L(h_1)=L'$, it follows that
	for each $1\leq i\leq m$, the $2m$ commutators $[h_1,x_i]$ and $[h_1,y_i]$ are all linearly independent elements of $\gamma_3(L)=Z(L)$ (see Proposition~\ref{dimension_center}). Thus without loss of generality, we can take $[h_1,x_i]=z_i$  and $[h_1,y_i]=z_{m+i}$  for $1
	\le i \le m$. This completes the proof.
	\end{proof}

	\begin{lemma}\label{suitable_generators}
		The generators $x_j$'s and $y_j$'s of $L$ (as in Lemma~\ref{presentation_class3_(0,m)_Liealgebra}) can be chosen such that
		$$[x_1,x_j]=[y_1,y_j]=0 \hskip5mm (2 \le j \le  m).$$
	\end{lemma}

	\begin{proof}
		Let $j$ be such that $2\le j\le m$. Since $x_1,x_j\notin L'$ and
		$[x_1,x_j]\in Z(L)$, by Corollary
		\ref{se4-cor1}, there exists $h'_j\in L'$ such that $[x_1,x_j+h'_j]=0$. Thus replacing $x_j$ by $x_j+h'_j$, we can assume that $[x_1,x_j]=0$. Similarly we can choose $y_j$'s such that $[y_1,y_j]=0$.
	\end{proof}

	Now we will determine the constants in the presentation of $L$ (as in Lemma~\ref{presentation_class3_(0,m)_Liealgebra}). We will see that these constants can be determined uniquely in terms of the structure constants $\kappa_{i,j,l}$'s of $\mathrm{U}_3(q^m)$. We start with the constants $\gamma$'s and $\delta$'s.
	
	\begin{theorem}\label{uniqueness_gamma_delta} 
		Let $L$ be as in Lemma~\ref{presentation_class3_(0,m)_Liealgebra}. For $1\leq i,j\leq m$, we have
		$$ [h_i,x_j]=\sum_{t=1}^{m}\kappa_{i,j,t}z_t \text{  and  }
		[h_i,y_j]=\sum_{t=1}^{m}\kappa_{i,j,t}z_{m+t}.$$
		In particular, $\gamma$'s and $\delta$'s in the relations (R3) and
		(R4) of Lemma~\ref{presentation_class3_(0,m)_Liealgebra} are uniquely determined by the structure constants $\kappa_{i,j,l}$ of ${\rm U}_3(q^m)$.
	\end{theorem}

	\begin{proof}
		We have $[x_1,y_i]=h_i$. By the previous lemma, $[x_1,x_j] = 0 \in Z(L)$ and hence using \eqref{eqse3-5a}, we get
		$$[h_i,x_j]=[[x_1,y_i], x_j]=[[x_j,y_i],x_1].$$
		
		Analogous to the relation $[x_1,y_i]=h_i$, it can also be shown that $y_i$'s can be chosen such that $[x_i,y_1]=h_i$. By (R8) of Lemma~\ref{presentation_class3_(0,m)_Liealgebra}, we get that $[x_j,y_i]=\sum\limits_{t=1}^{m}\kappa_{i,j,t}h_t + Z(L)$. We can now conclude that for $1\leq i,j\leq m$, 
		$$[x_j,y_i]=\sum_{t=1}^{m}\kappa_{i,j,t}[x_t,y_1]=\left[\sum_{t=1}^{m}\kappa_{i,j,t}x_t,y_1\right]$$ 
		modulo the center $Z(L)$.
	  	Once again using previous lemma, we have that 
	  	$$\left[x_1, \sum_{t=1}^{m}\kappa_{i,j,t}x_t\right] =0 \in Z(L).$$
	  	Using \eqref{eqse3-5a} we get, 
		
		\begin{eqnarray*}
		[h_i,x_j]&=&[[x_j,y_i],x_1] = \left[\left[\sum_{t=1}^{m}\kappa_{i,j,t}x_t,y_1\right],x_1\right] \\ &=&
		\left[[x_1,y_1], \sum_{t=1}^{m}\kappa_{i,j,t}x_t\right]=\sum_{t=1}^{m}\kappa_{i,j,t}[h_1,x_t]=\sum_{t=1}^{m}\kappa_{i,j,t}z_t.
		\end{eqnarray*}
		The last equality follows from (R7) of Lemma~\ref{presentation_class3_(0,m)_Liealgebra}. This proves one of the assertions. The other assertion follows from similar computations.
	\end{proof}

	\begin{lemma}\label{separation_of_center}
		In the Lie algebra $L$, for $1 \le i, j \le m$, we have 
		$$[x_i,x_j] \in \gen{z_1,\ldots, z_m},\,\,\,\, \mbox{ and } \,\,\,\,
		[y_i,y_j] \in \gen{z_{m+1},\ldots, z_{2m}}.$$
	\end{lemma}

	\begin{proof}
		Fix $i,j$ with $1\le i,j\leq m$. Since $[x_i,x_j]\in Z(L)$, by
		Corollary  \ref{se4-cor1}, there exists $h\in L'$ (depending on $x_i,x_j$) such that
		$[x_i,x_j+h]=0$. Then
		$$[x_i,x_j]=-[x_i,h]=[h,x_i].$$
		Since $L'=\langle h_1,h_2,\ldots,h_m,Z(L)\rangle$, by Theorem \ref{uniqueness_gamma_delta}, $[h,x_i]\in\langle z_1,\ldots, z_m\rangle$, proving the first assertion. The second assertion follows similarly.
	\end{proof}

	\begin{theorem}\label{lambdaequals0}
		We have $\lambda_{i,j,l}=0$ for all $1\leq i,j\leq m$ and $1\leq l\leq 2m$. In other words, the relation (R8) in Lemma~\ref{presentation_class3_(0,m)_Liealgebra} reduces to the relation $[x_i,y_j]=\sum_{t=1}^{m}\kappa_{i,j,t}h_t$ for every $1\leq i,j\leq m$.
	\end{theorem}

	\begin{proof}
		From relation (R8) in Lemma~\ref{presentation_class3_(0,m)_Liealgebra}, we have
		$$[x_i,y_j]=\sum_{t=1}^{m}\kappa_{i,j,t}h_t+ \sum_{l=1}^{2m} \lambda_{i,j,l}z_l.$$
		Since $[x_1,y_i]=h_i$, the above can be written as 
		$$[x_i,y_j]=[x_1,\sum_{t=1}^{m}\kappa_{i,j,t}y_t]+ \sum_{l=1}^{2m} \lambda_{i,j,l}z_l.$$
		Let $y=\sum_{t=1}^{m}\kappa_{i,j,t}y_t$. Then we can conclude that $[x_i,y_j]-[x_1,y]\in Z(L)$. By Lemma~\ref{suitable_generators}, for every $r\in \F_q\setminus \{0\}$, we get $[x_1+ry_j,x_i+ry]=0$. By Corollary~\ref{se4-cor1}, there exists $\theta(r)\in L'$ (depending on $r$) such that $[x_1+ry_j,x_i+ry+\theta(r)]=0.$ Once again using Lemma~\ref{suitable_generators}, we have
		\begin{eqnarray*}
			&& [x_1+ry_j,x_i+ry+\theta(r)]=0 \implies r[x_1,y]+[x_1,\theta(r)]+r[y_j,x_i]+r[y_j,\theta(r)]=0 \\
			&& \implies r([x_i,y_j]-[x_1,y])=[x_1,\theta(r)] + r[y_j,\theta(r)]\\ &&\implies \sum_{l=1}^{2m}  r\lambda_{i,j,l}z_l=[x_1,\theta(r)] + r[y_j,\theta(r)].
		\end{eqnarray*}
		Using Theorem~\ref{uniqueness_gamma_delta}, we can conclude that
		\begin{equation}\label{maineq}
			[x_1,\theta(r)]=r\sum_{l=1}^{m}  \lambda_{i,j,l}z_l,\;\;\text{ and }\;\;[y_j,\theta(r)]=\sum_{l=m+1}^{2m}  \lambda_{i,j,l}z_l.
		\end{equation}
		Since $[h_i,x_1]=z_1$, by Theorem~\ref{uniqueness_gamma_delta}, we have
		\begin{eqnarray*}
			\sum_{l=1}^{m}  \lambda_{i,j,l}z_l&=&\sum_{l=1}^{m}  \lambda_{i,j,l}[h_l,x_1]\\ &=& \sum_{l=1}^{m}[\lambda_{i,j,l}h_l,x_1]=[\tilde{h},x_1]
		\end{eqnarray*}
		where $\tilde{h}=\sum_{l=1}^{m}\lambda_{i,j,l}h_l$. From the first relation in \eqref{maineq} we get $[x_1,\theta(r)]=r[\tilde{h},x_1]$. This gives $[x_1,\theta(r)-r\tilde{h}]=0$. Thus, $\theta(r)-r\tilde{h}\in C_L(x_1)\cap L'=Z(L)$ (by Corollary~\ref{centralizer_info_L}). Putting this in the second relation of \eqref{maineq} we get
		$$\sum_{l=m+1}^{2m}  \lambda_{i,j,l}z_l=[y_j,\theta(r)]=r[y_j,\tilde{h}].$$
		Note that $\tilde{h}$ is not dependent on the choice of $r$. From the above equation we get
		$$(q-1)\sum_{l=m+1}^{2m}  \lambda_{i,j,l}z_l=\sum_{r\in \F_q\setminus \{0\}}r[y_j,\tilde{h}]=0.$$
		The last equality follows since $\sum\limits_{r\in \F_q\setminus \{0\}}r=0$ when $q\geq 3$. Note that the LHS of the above equation does not vanish since $q$ is a prime-power and hence $q-1\neq 0$ in $\F_q$.
		Since $z_1,z_2,\ldots,z_{2m}$ are linearly independent we conclude that $\lambda_{i,j,l}=0$ for all $1\leq i,j\leq m$ and $m+1\leq l\leq 2m$. Thus the relation in (R8) reduces to 
		$$[x_i,y_j]=\sum_{t=1}^{m}\kappa_{i,j,t}h_t+ \sum_{l=1}^{m} \lambda_{i,j,l}z_l.$$
		To prove that $\lambda_{i,j,l}=0$ for $1\leq i,j,l \leq m$ we argue along similar lines. Once again taking $y=\sum_{t=1}^{m}\kappa_{i,j,t}y_t$ we have $[x_i,y_j]-[x_1,y]\in Z(L)$.  By Lemma~\ref{suitable_generators}, for every $r\in \F_q\setminus\{0\}$, we get $[y_j+rx_1,y+rx_i]=0$. By Corollary~\ref{se4-cor1}, there exists $\theta(r)\in L'$ (depending on $r$) such that $[y_j+rx_1,y+rx_i+\theta(r)]=0.$ Using Lemma~\ref{suitable_generators} we have
		\begin{eqnarray*}
			&& \;\;\;\;\;\;\;\;\;\;[y_j+rx_1,y+rx_i+\theta(r)]=0 \\ && \implies r[y_j,x_i]+[y_j,\theta(r)]+r[x_1,y]+r[x_1,\theta(r)]=0 \\ &&
			\implies r([x_i,y_j]-[x_1,y])=r[x_1,\theta(r)]+[y_j,\theta(r)] \\ &&
			\implies r\sum_{l=1}^{m}\lambda_{i,j,l}z_l=r[x_1,\theta(r)]+[y_j,\theta(r)].
		\end{eqnarray*}
		By Theorem~\ref{uniqueness_gamma_delta}, we conclude that $[y_j,\theta(r)]=0$ which implies that $\theta(r)\in C_L(y_j)\cap L'=Z(L)$ (by Corollary~\ref{centralizer_info_L}). This gives $[x_1,\theta(r)]=0$. Consequently, $\sum_{l=1}^{m}\lambda_{i,j,l}z_l=0$, whence we conclude that $\lambda_{i,j,l}=0$ for every $1\leq i,j,l\leq m$. This completes the proof.
	\end{proof}
	
	Finally we deal with the remaining two pairs of constants $\alpha$'s and $\beta$'s.
	\begin{theorem}\label{alpha_beta_0}
		 In the relations (R5)-(R6) of Lemma~\ref{presentation_class3_(0,m)_Liealgebra}, we have $\alpha_{i,j,l}=0$  and $\beta_{i,j,l}=0$ for every $1 \le
		i, j \le m$, $1 \le l \le 2m$. In other words, $[x_i,x_j]=[y_i,y_j]=0$ for every $1\leq i,j\leq m$.
	\end{theorem}
	
	\begin{proof}
		Note that from the previous theorem, $[x_i,y_j]=[x_j,y_i]$ for all $1\leq i,j\leq m$. Let $r\in \F_q\setminus\{0\}$. We get $[x_i+ry_i,x_j+ry_j]\in Z(L)$. By Corollary~\ref{se4-cor1}, we conclude that there exists $\theta(r)\in L'$ such that $[x_i+ry_i, x_j+ry_j+\theta(r)]=0$. This yields 
		$$[x_i,x_j]+[x_i,\theta(r)]+r^2[y_i,y_j]+r[y_i,\theta(r)]=0.$$
		From Theorem~\ref{uniqueness_gamma_delta} and Lemma~\ref{separation_of_center}, we conclude that
		\begin{equation}\label{se5_5.6_eqn1}
			[x_i,x_j]+[x_i,\theta(r)]=0\;\text{ and }\; r[y_i,y_j]+[y_i,\theta(r)]=0.
		\end{equation}
		
		Since $[h_1,x_i], [h_2,x_i],\ldots, [h_m,x_i]$ are independent and belong to $\langle z_1,z_2,\ldots,z_m\rangle$ (see Theorem \ref{uniqueness_gamma_delta}), we get
		\begin{align*}
			\langle z_1,z_2,\ldots, z_m\rangle =\langle [h_1,x_i], [h_2,x_i],\ldots.
			[h_m,x_i]\rangle.
		\end{align*}
		Thus, by Lemma~\ref{separation_of_center}, we conclude that $[x_i,x_j]=[x_i,\theta']$ for some $\theta'\in L'$. From the first equation in \eqref{se5_5.6_eqn1} we have $[x_i,\theta'+\theta(r)]=0$, whence it follows that $\theta'+\theta(r)\in Z(L)$. Note that $\theta'$ does not depend on the value of $r$. From the second equation in \eqref{se5_5.6_eqn1} we get $[y_i,\theta']=r[y_i,y_j]$ for every $r\in \F_q^{*}$. Hence
		$$(q-1)[y_i,\theta']=\sum_{r\in \F_q\setminus\{0\}}r[y_i,y_j]=0,$$
		since $\sum\limits_{r\in \F_q\setminus\{0\}}r=0$ when $q\geq 3$. It follows that $[y_i,\theta']=0$ and hence $\theta'\in C_L(y_i)\cap L'=Z(L)$. Thus $[x_i,x_j]=[x_i,\theta']=0$ and the result follows. The proof for $\beta$'s follow along similar lines.
	\end{proof}
	
	We are now ready to prove Theorem~\ref{even_dimension_three_step_nilpotent}.
	\begin{proof}[\textbf{Proof of 	Theorem~\ref{even_dimension_three_step_nilpotent}}]
		Let $L$ be a 3-step nilpotent stem Lie algebra over $\F_q$ of breadth type $(0,2m)$ where $m\geq 1$. Note that when $m=1$, that is, when $L$ is 3-step nilpotent stem Lie algebra of breadth type $(0,2)$ (see Theorem 5.1, \cite{kns}), then $L$ is isomorphic to the Lie algebra $\mathcal{V}$ defined by the presentation 
		$$\mathcal{V}=\langle x_1,x_2,y,z_1,z_2\mid [x_1,x_2]=y,\;[x_1,y]=z_1,\; [x_2,y]=z_2\rangle.$$
		Thus, there is a unique such stem Lie algebra. Now we assume $m\geq 2$. Up to isomorphism, Theorem~\ref{uniqueness_gamma_delta}, Theorem~\ref{lambdaequals0} and Theorem~\ref{alpha_beta_0} uniquely determine the presentation of $L$ (once we fix the structure constants of $\mathrm{U}_3(q^m)$ which is given in Lemma~\ref{presentation_class3_(0,m)_Liealgebra}). Thus, $L$ is isomorphic to a Lie algebra $\mathcal{V}$ which is  given as follows:
		\begin{align*} 
			&\mathcal{V}=\Big{\langle}  x_1,\ldots, x_m,y_1,\ldots, y_m, h_1, h_2,\ldots, h_m, z_1,
			z_2, \ldots, z_{2m}~~ \Big{|} [h_i,x_j]=\sum_{t=1}^{m}\kappa_{i,j,t}z_t,\\
			& [h_i,y_j]=\sum_{t=1}^{m}\kappa_{i,j,t}z_{m+t},\; [x_i,y_j]=\sum_{t=1}^{m}\kappa_{i,j,t}h_t \;(1\leq i,j\leq m) \Big{\rangle}; 
		\end{align*}
		where   $\kappa_{i,j,l}$,  $1 \le i, j, l \le m$, are the structure constants of ${\rm U}_3(q^m)$.
		Note that in the above presentation, the Lie brackets we do not write are zero. Once again we get a unique Lie algebra up to isomorphism. The proof is now complete.
	\end{proof}
	
	\begin{proof}[\textbf{Proof of Corollary~\ref{Corr}}]
		The proof follows from Theorem~\ref{even_dimension_three_step_nilpotent} and Theorem~\ref{example_3_step_nilpotent_over_finite_fields}.
	\end{proof}

	\bibliographystyle{amsalpha}
	\bibliography{References}	

\providecommand{\bysame}{\leavevmode\hbox to3em{\hrulefill}\thinspace}
\providecommand{\MR}{\relax\ifhmode\unskip\space\fi MR }
\providecommand{\MRhref}[2]{%
  \href{http://www.ams.org/mathscinet-getitem?mr=#1}{#2}
}
\providecommand{\href}[2]{#2}
\begin{thebibliography}{CdGS12}

\bibitem[BI03]{bi}
Yiftach Barnea and I.~M. Isaacs, \emph{Lie algebras with few centralizer
  dimensions}, J. Algebra \textbf{259} (2003), no.~1, 284--299. \MR{1953720}

\bibitem[CdGS12]{sgs}
Serena Cical\`o, Willem~A. de~Graaf, and Csaba Schneider, \emph{Six-dimensional
  nilpotent {L}ie algebras}, Linear Algebra Appl. \textbf{436} (2012), no.~1,
  163--189. \MR{2859920}

\bibitem[CW99]{cw}
M.~Cordero and G.~P. Wene, \emph{A survey of finite semifields}, vol. 208/209,
  1999, Combinatorics (Assisi, 1996), pp.~125--137. \MR{1725526}

\bibitem[dG07]{dg}
Willem~A. de~Graaf, \emph{Classification of 6-dimensional nilpotent lie
  algebras over fields of characteristic not 2}, J. Algebra \textbf{309}
  (2007), no.~2, 640--653.

\bibitem[JP21]{jnp}
Peyman Johari, Farangis;~Niroomand and Mohsen Parvizi, \emph{Nilpotent lie
  algebras of class 4 with the derived subalgebra of dimension 3}, Comm.
  Algebra \textbf{49} (2021), no.~2, 521--532.

\bibitem[Kan06]{ka}
William~M. Kantor, \emph{Finite semifields}, Finite geometries, groups, and
  computation, Walter de Gruyter, Berlin, 2006, pp.~103--114. \MR{2258004}

\bibitem[KMS15]{kms}
Borworn Khuhirun, Kailash~C. Misra, and Ernie Stitzinger, \emph{On nilpotent
  {L}ie algebras of small breadth}, J. Algebra \textbf{444} (2015), 328--338.
  \MR{3406181}

\bibitem[KNS21]{kns}
Rijubrata Kundu, Tushar~Kanta Naik, and Anupam Singh, \emph{Nilpotent lie
  algebras of breadth type $(0, 3) $}, arXiv preprint arXiv:2111.04968 (2021).

\bibitem[Knu65]{kn}
Donald~E. Knuth, \emph{Finite semifields and projective planes}, J. Algebra
  \textbf{2} (1965), 182--217. \MR{175942}

\bibitem[KS15]{ks}
Norbert Knarr and Markus~J. Stroppel, \emph{Heisenberg groups, semifields, and
  translation planes}, Beitr. Algebra Geom. \textbf{56} (2015), no.~1,
  115--127. \MR{3305438}

\bibitem[Lew17]{le}
Mark~L Lewis, \emph{Semi-extraspecial groups}, Southern Regional Algebra
  Conference, Springer, 2017, pp.~219--237.

\bibitem[NKY20]{nky}
Tushar~Kanta Naik, Rahul~Dattatraya Kitture, and Manoj~K. Yadav, \emph{Finite
  {$p$}-groups of nilpotency class 3 with two conjugacy class sizes}, Israel J.
  Math. \textbf{236} (2020), no.~2, 899--930. \MR{4093905}

\bibitem[Rem17]{re}
Elisabeth Remm, \emph{Breadth and characteristic sequence of nilpotent {L}ie
  algebras}, Comm. Algebra \textbf{45} (2017), no.~7, 2956--2966. \MR{3594570}

\bibitem[Rom89]{ro}
Mustapha Romdhani, \emph{Classification of real and complex nilpotent {L}ie
  algebras of dimension {$7$}}, Linear and Multilinear Algebra \textbf{24}
  (1989), no.~3, 167--189. \MR{1007253}

\bibitem[Sch05]{sc}
Csaba Schneider, \emph{A computer-based approach to the classification of
  nilpotent {L}ie algebras}, Experiment. Math. \textbf{14} (2005), no.~2,
  153--160. \MR{2169519}

\bibitem[See93]{se}
Craig Seeley, \emph{{$7$}-dimensional nilpotent {L}ie algebras}, Trans. Amer.
  Math. Soc. \textbf{335} (1993), no.~2, 479--496. \MR{1068933}

\bibitem[SWK21]{swk}
Songpon Sriwongsa, Keng Wiboonton, and Borworn Khuhirun, \emph{Characterization
  of nilpotent lie algebras of breadth 3 over finite fields of odd
  characteristic}, Journal of Algebra \textbf{586} (2021), 935--972.

\bibitem[Ver87]{ve}
Libero Verardi, \emph{Semi-extraspecial groups of exponent {$p$}}, Ann. Mat.
  Pura Appl. (4) \textbf{148} (1987), 131--171. \MR{932762}

\end{thebibliography}
\end{document}